\documentclass[10pt,headsepline, a4paper,cleardoubleempty]{scrartcl}

\usepackage[utf8]{inputenc}
\usepackage[intlimits]{amsmath}
\usepackage{amsfonts}
\usepackage{amssymb}
\usepackage[amsmath, thmmarks]{ntheorem}
\usepackage{enumerate}

\newcommand{\R}{\mathbb{R}}
\newcommand{\Z}{\mathbb{Z}}
\newcommand{\Q}{\mathbb{Q}}
\newcommand{\N}{\mathbb{N}}
\newcommand{\C}{\mathbb{C}}
\renewcommand{\P}{\mathbb{P}}
\newcommand{\transpose}{^{\operatorname{t}}}
\renewcommand{\sl}{\operatorname{SL}(2,\Z)}
\newcommand{\slR}{\operatorname{SL}(2,\R)}
\newcommand{\ii}{\mathrm{i}}
\renewcommand{\Im}{\operatorname{Im}}

\newcommand{\Norm}{\operatorname{N}}
\newcommand{\ord}{\operatorname{ord}}

\makeatother
\newtheorem{vorlage}{}[section]
\newtheorem{prop}[vorlage]{Proposition}
\newtheorem{lemma}[vorlage]{Lemma}
\newtheorem{cor}[vorlage]{Corollary}

\newtheorem{theorem}[vorlage]{Theorem}
\theorembodyfont{\rmfamily}
\newtheorem{definition}[vorlage]{Definition}
\newtheorem{rem}[vorlage]{Remark}

\theoremstyle{nonumberbreak}
\theoremsymbol{\ensuremath{\square}}
\newtheorem{proof}{Proof}

\title{Calculating the Fourier Coefficients of Jacobi--Eisenstein series}
\author{Martin Woitalla}

\begin{document}

\maketitle

\begin{abstract}
In this text we generalize the classical Jacobi Eisenstein series as they were
discussed by Eichler and Zagier to arbitrary lattices. We use two 
different methods to derive the general Fourier expansion. The last two sections give formulas in the case where the lattice is an even unimodular lattice or $L$ is of the type $NA_{1}$.
\end{abstract}
\section{Jacobi-Eisenstein series and Jacobi group}
Let $L$ be an even integral lattice of rank $N$ where we indentify $L$
with $\Z^{N}$ together with the bilinear form $(\cdot,\cdot)_{L}$ which
has gram matrix $S$ with respect to the standard basis of $\Z^{N}$ which we
denote by $\mathfrak{e}_{1},\dots,\mathfrak{e}_{N}$. So $S$
is an even and positive definite matrix. We will also write $(\cdot,\cdot)_{S}$
for the bilinear form and will drop the subscript $S$, when it is obvious which
lattice is involved.
\vspace{5mm}

\noindent We define
\[L^{\vee}=\left\{\lambda\in L\otimes
\Q\,|\,(\lambda,l)\in\Z\,\,\text{for all}\,\, l\in L\right\}\,.\]
The set $L^{\vee}$ is a lattice again and is called the \textit{dual lattice}
of $L$. Since $L$ is integral we have $L\subseteq
L^{\vee}$.The \textit{discriminant
group} of the lattice is given by $L^{\vee}/L$ and is a finite abelian group of
order $\det(S)$. Associated to the lattice is the quadratic form
\[L^{\vee}/L\times L^{\vee}/L\to\Q/2\Z\quad,\quad
(\lambda+L,\lambda+L):=(\lambda,\lambda)+2\Z\]
which is called the \textit{discriminant form} of the lattice $L$.
The corresponding bilinear form is given by 
\[L^{\vee}/L\times L^{\vee}/L\to\Q/\Z\quad,\quad
(\lambda+L,\mu+L):=(\lambda,\mu)+\Z\,.\]
The real Heisenberg group
$H(L\otimes\R)$ is the group which is generated by the elements
\[[x,y:r]:=\begin{pmatrix}
            1&0&y\transpose S&\frac{1}{2}(x,y)-r&\frac{1}{2}(y,y)\\
	    0&1&x\transpose S&\frac{1}{2}(x,x)&\frac{1}{2}(x,y)+r\\
	    0&0&I_{N}&x&y\\
	    0&0&0&1&0\\
	    0&0&0&0&1
           \end{pmatrix}\]
where $x,y\in\R^{N}$ and $r\in\R$. The group law of $H(L\otimes\R)$ is 
\[[x,y:r]\cdot
[x',y':r']=[x+x',y+y',r+r'+\frac{1}{2}((x,y')-(x',y))]\,.\]
We observe that the projection
\[H(L\otimes\R)\to (L\otimes \R)\times (L\otimes \R)\quad,\quad
[x,y,r]\mapsto (x,y)\]
has kernel $\left\{[0,0:r];r\in\R\right\}$ which equals the centre of
$H(L\otimes\R)$. Therefore the Heisenberg group is the central extension of
$\R$ with $(L\otimes \R)\times (L\otimes \R)$ 
\[0\to \R\to H(L\otimes\R)\to(L\otimes \R)\times (L\otimes \R)\to 0\,.\]
We have an action of $\slR$ on the Heisenberg group which is given by
\[\slR\circlearrowright H(L\otimes\R)\quad,\quad
A.[x,y:r]:=\{A\}[x,y:r]\{A^{-1}\}=[dx-cy,-bx+ay:r]\,.\]
As usual we denote the 1-sphere by
\[S^{1}=\left\{\zeta\in\C;|\zeta|=1\right\}\,.\]
Following Eichler Zagier we consider the group 
\[C_{\Z}=\left\{[0,0:r];r\in\Z\right\}\]
which is a subgroup of the centre of $H(L\otimes\R)$. This gives us the
following description of the factor group 
\begin{lemma}\label{lemma:central extension factor group}
 There is a central extension of $H(L\otimes\R)/C_{\Z}$ given by
\[1\to S^{1}\to H(L\otimes\R)/C_{\Z}\to (L\otimes \R)\times (L\otimes \R)\to
0\,.\]
\end{lemma}
\begin{proof}
We define an embedding $\iota:S^{1}\to
H(L\otimes\R)/C_{\Z}\,,\,e^{2\pi\ii\kappa}\mapsto [0,0:\kappa]$. This is
obviously well-defined and injective and since $\iota(S^{1})$ is contained  in
the centre of the Heisenberg group we also have $\iota(S^{1})\subseteq
C\left(H(L\otimes\R)/C_{\Z}\right)$. If we define the projection as above we
obtain that the sequence is exact. 
\end{proof}
The factor group inherits the $\slR$-action from $H(L\otimes\R)$ so that we can
define
\[G^{J}:=\slR\ltimes H(L\otimes\R)/C_{\Z}\,.\] 
According to Lemma \ref{lemma:central extension factor group} we 
write $\left\{A,[x,y],\zeta\right\}$ for the elements of $G^{J}$ where
$A\in\slR\,,\, x,y\in L\otimes \R$ and $\zeta\in S^{1}$ and the group law reads
\[\begin{aligned}
   &\left\{A,[x,y],\zeta\right\}\cdot
\left\{A',[x',y'],\zeta'\right\}&\\&=\left\{AA',[d'x-c'y,-b'x+a'y]+[x',y'],
\zeta\zeta'\cdot e^{\pi\ii ((d'x-c'y,y')-(x',-b'x+a'y))}\right\}& 
  \end{aligned}\]
We now consider the discrete subgroup
\[\Gamma^{J}:=\sl\ltimes (L\times
L)=\left<\left\{A,[x,y],1\right\}|A\in\sl,x,y\in L\right>\leq G^{J}\,.\]
This group corresponds to the classical Jacobi group from Eichler and Zagier.
Let $\mathcal{H}$ be the usual upper half plane in $\C$ and
$\phi:\mathcal{H}\times (L\otimes \C)\to \C$ be a function. Then we define
\begin{equation}\label{Jacobi law SL}
\left(\phi\left.\right|_{k,m}\begin{pmatrix}
                               a&b\\c&d
                              \end{pmatrix}
\right)(\tau,\mathfrak{z}):=(c\tau+d)^{-k}e^{-\pi\ii
m\frac{c(\mathfrak{z},\mathfrak{z})}{c\tau
+d}}\phi\left(\frac{a\tau+b}{c\tau+d},\frac{\mathfrak{z}}{c\tau
+d}\right) 
\end{equation}
$\text{for all} \begin{pmatrix}                              
a&b\\c&d\end{pmatrix}\in\sl$ and
\begin{equation}\label{Jacobi law Z2}
\left(\phi\left.\right|_{m}[x,y]\right)(\tau,\mathfrak{z}):=e^{\pi\ii
m((x,x)\tau+2(x,\mathfrak{z}))}\phi\left(\tau,\mathfrak{z}
+x\tau+y\right) 
\end{equation}
for all $x,y\in L$. One easily shows that (\ref{Jacobi law SL}) and (\ref{Jacobi
law Z2}) jointly define an action of $\Gamma^{J}$ on $\C^{\mathcal{H}\times
(L\otimes \C)}$

We want
to remind to the definition of a Jacobi form.
\begin{definition}
A holomorphic \textit{Jacobi form} of index $m$ and weight $k$ ($k,m\in\N$) is a
holomorphic
function on $\mathcal{H}\times (L\otimes \C)$ satisfying 
\[\phi\left.\right|_{k,m}g=\phi\quad\text{for all}\quad g\in
\Gamma^{J}\]
and having a Fourier expansion 
\[\phi(\tau,\mathfrak{z})=\sum_{\substack{n\in\Z,l\in
L^{\vee}\\2nm-(l,l)\geq 0}}{f(n,l)e^{2\pi\ii
n\tau+2\pi\ii(l,\mathfrak{z})}}\,.\] 
\end{definition}
We want to define Jacobi-Eisenstein series. Therefore we introduce the subgroup
\[\Gamma_{\infty}^{J}=\left\{g\in\Gamma^{J};1\left.\right|_{k,m}g=1\right\}\]
 We now assume that $k$ is even. Since every element in $\Gamma^{J}$ can be
written in the form $\left\{I,[x,y]\right\}\times \left\{A,[0,0]\right\}$ where
$x,y\in L$ and $A\in\sl$  we obtain
\[\begin{aligned}  
&\left.\left(1\left.\right|_{k,m}\left\{I,[x,y]\right\}\right)\right|_{k,m}
\left\{A,[0,0]\right\}(\tau,\mathfrak{z})&\\
&=\left.e^{\pi\ii
m((x,x)\tau+2(x,\mathfrak{z}))}\right|_{k,m}A&\\
&=(c\tau+d)^{-k}\cdot e^{-\pi\ii
m\frac{c(\mathfrak{z},\mathfrak{z})}{c\tau
+d}}\cdot e^{\pi\ii
m((x,x)\frac{a\tau+b}{c\tau+d}+2(x,\frac{\mathfrak{z}}{c\tau+d}))}&
  \end{aligned}\]
We put $(\tau,\mathfrak{z})=(2\ii,0)$ and hence obtain with
$\varphi=\operatorname{arg}(2c\ii+d)$
\[\begin{aligned}  
&\left.\left(1\left.\right|_{k,m}\left\{I,[x,y]\right\}\right)\right|_{k,m}
\left\{A,[0,0]\right\}(2i,0)&\\
&=(2c\ii+d)^{-k}\cdot e^{-\pi\ii
m\frac{c(0,0)}{2c\ii
+d}}\cdot e^{\pi\ii
m((x,x)\frac{2a\ii+b}{2c\ii+d}+2(x,\frac{0}{2c\ii+d}))}&\\
&=(4c^{2}+d^{2})^{-k/2}\cdot e^{-\ii k\varphi}\cdot e^{\ii\pi
m\frac{(2i+4ac+bd)(x,x)}{4c^{2}+d^{2}}}&\\
&=(4c^{2}+d^{2})^{-k/2}e^{-\frac{\pi m(x,x)}{4c^{2}+d^{2}}}\cdot e^{-\ii
k\varphi}\cdot e^{\ii\pi
m\frac{(4ac+bd)(x,x)}{4c^{2}+d^{2}}}&
  \end{aligned}\]
So one has
\[\left\{I,[x,y]\right\}
\left\{A,[0,0]\right\}\in\Gamma_{\infty}^{J}\iff 1=
(4c^{2}+d^{2})^{-k/2}e^{-\frac{\pi m(x,x)}{4c^{2}+d^{2}}}\cdot e^{-\ii
k\varphi}\cdot e^{\ii\pi
m\frac{(4ac+bd)(x,x)}{4c^{2}+d^{2}}}\]
This immediately implies $x=0$ and $c=0$. This gives us the alternative
description \[\Gamma_{\infty}^{J}=\left\{\left\{\pm\begin{pmatrix}
                                     1&n\\0&1
                                    \end{pmatrix}
,[0,y]\right\};y\in L,n\in\Z\right\}\,.\]
One has the following decomposition into cosets
\begin{equation}\label{decompostion in cosets}
 \begin{aligned}
  \Gamma^{J}&=&&\bigcup_{\substack{c,d\in\Z\,,\, c\geq 0\\ (c,d)\neq(0,-1)\\
\gcd(c,d)=1}}{\bigcup_{x\in L}{\Gamma_{\infty}^{J}}\left\{\begin{pmatrix}
                              a&b\\c&d
                             \end{pmatrix},[x,0]
 \right\}}&\\
&=&&\bigcup_{\substack{c,d\in\Z\,,\,c\geq 0\\(c,d)\neq(0,-1)\\
\gcd(c,d)=1}}{\bigcup_{x\in L}{\Gamma_{\infty}^{J}}\left\{\begin{pmatrix}
                              1&0\\0&1
                             \end{pmatrix},[x,0]
 \right\}\cdot \left\{\begin{pmatrix}
                              a&b\\c&d
                             \end{pmatrix},[0,0]
 \right\}}&
 \end{aligned}
\end{equation}
where $a,b$ are chosen such that $\begin{pmatrix}
                              a&b\\c&d
                             \end{pmatrix}\in\sl\,.$  

\begin{definition}\label{Def:Jacobi E Series}
Let $k,m\in\N$, $k$ even. The \textit{Jacobi-Eisenstein series} of weight $k$
and index
$m$ is given by 
\[E_{k,m}:=\sum_{g\in\Gamma_{\infty}^{J}\backslash
\Gamma^{J}}{\left.1\right|_{k,m}g}\] 
\end{definition}
We want to describe the region of convergence for the Jacobi-Eisenstein series. The next
result can be found in section 2 of \cite{Ziegler}.
\begin{prop}\label{prop:convergence JE series}
 For $k>2+N$ even the Jacobi-Eisenstein series $E_{k,m}$ converges normally on
$\mathcal{H}\times (L\otimes \C)$. The convergence is uniform on vertical
strips of the form
\[V(\delta):=\left\{(\tau,\mathfrak{z})\in\mathcal{H}\times (L\otimes
\C)\,;\,\tau=x+\ii y\,,\,y\geq \delta\,,\,x^{2}\leq
\delta^{-1}\,,\,\mathfrak{z}^{\ast}\mathfrak{z}\leq \delta^{-1}\right\}\,.\]
\end{prop}
As in \cite{Ziegler} we introduce theta series. 
\begin{definition}
 Let $S\in\Z^{N\times N}$ be symmetric and positive definite. Let
$a,b\in\Q^{N}$, $\tau\in\mathcal{H}$ and $\mathfrak{z}\in L\otimes \C$. The
\textit{theta series} of charaacteristic $S,a,b$ is given by
\[\Theta_{S,a,b}(\tau,\mathfrak{z}):=\sum_{\lambda\in\Z^{N}}{\operatorname{\exp}
(\pi\ii\left\{(\lambda+a,\lambda+a)_{S}\,\tau+2(\lambda+a,\mathfrak{z}+b)_{S}
\right\} ) } \] 
\end{definition}
\begin{rem}
 If we use the notation of \cite{KriegMFQuat} we get 
\[\Theta_{S,a,b}(\tau,\mathfrak{z})=\Theta_{a,S(\mathfrak{z}+b)}(\tau,S,\Z^{N}) 
\,.\]
Thus $\Theta_{S,a,b}(\tau,\mathfrak{z})$ converges absolutely and locally
uniformly in dependence on $S,a,b,\tau$ and $\mathfrak{z}$. Moreover
$\Theta_{S,a,b}(\tau,\mathfrak{z})$ considered as a function of $\tau$ and
$\mathfrak{z}$ turns out to be holomorphic on $\mathcal{H}\times (L\otimes
\C)$.
\end{rem}

\section{Two formulae for the Fourier expansion of
$E_{k,m}$}\label{section:formulae for FJ coefficients}
In this section we derive the Fourier expansion of $E_{k,m}$ for arbitrary
lattices. Moreover we give two different formulae for the Fourier coefficients
which are contained in Theorem  \ref{theorem:Fourier expansion according to EZ}
for arbitrary lattices and in Theorem \ref{theorem: Fourier expansion for m
equals 1} for the special case if $m$ equals 1 and $\operatorname{rank}(L)$ is
even. In both formulae certain Dirichlet series appear where the coefficients
are given by certain representation numbers. The basic problem is to evaluate
these Dirichlet series which is done for special types of lattices in the following chapters. The derivation of the first formula follows the methods presented in \cite{Eichler-Zagier} and can be considered as an analytic formula. The second formula is a generalization of the technique used in \cite{Eie-Krieg} where we make use of the theta transformation law. This one is more natural since it contains only intrinsic data of the lattice $L$.
\vspace{5mm}

\noindent We start with the coset system from (\ref{decompostion in cosets})
to derive an explicit description of $E_{k,m}$:
\[\begin{aligned}
E_{k,m}(\tau,\mathfrak{z})&=&&\frac{1}{2}\sum_{\substack{c,d\in\Z\\(c,d)=1}}{
\sum_ { x\in L}{\left.\left(\left.1\right|_{k,m}\left\{\begin{pmatrix}
                              1&0\\0&1
                             \end{pmatrix},[x,0]
 \right\}\right)\right|_{k,m} \left\{\begin{pmatrix}
                              a&b\\c&d
                             \end{pmatrix},[0,0]
 \right\}(\tau,\mathfrak{z})}} &\\
&=&&\frac{1}{2}\sum_{\substack{c,d\in\Z\\(c,d)=1}}{
\sum_ { x\in L}{(c\tau+d)^{-k} \exp\left(\pi\ii
m\frac{(x,x)(a\tau+b)+2(x,\mathfrak{z}
)-
c(\mathfrak{z},\mathfrak{z})}{c\tau
+d}\right)}}   &
  \end{aligned}\]
Next we split $E_{k,m}$ into the parts
\[E_{k,m}=E_{k,m}^{0}+E_{k,m}^{\ast}\]
where the first summands denotes the contribution from the summands with $c=0$
and the second one stands for the summands where $c\neq 0$. If $c=0$ we have
$d=\pm 1$. Since we have $-L=L$ it follows
\[\begin{aligned}
E_{k,m}^{0}(\tau,\mathfrak{z})&=&&\frac{1}{2}
\sum_{x\in L}{\exp\left(\pi\ii
m\left((x,x)\tau+2(x,\mathfrak{z})\right)\right)}+\frac{1}{2}
\sum_{x\in L}{\exp\left(\pi\ii
m\left((x,x)\tau-2(x,\mathfrak{z})\right)\right)}&\\
&=&&\sum_{x\in L}{\exp\left(\pi\ii
m\left((x,x)\tau+2(x,\mathfrak{z})\right)\right)}&\\
&=&&\Theta_{mS,0,0}(\tau,\mathfrak{z})\,.&
  \end{aligned}\]
If $c\neq 0$ we have
\[\begin{aligned}
   &\frac{(x,x)(a\tau+b)+2(x,\mathfrak{z}
)-
c(\mathfrak{z},\mathfrak{z})}{c\tau
+d}\qquad&\\
&=-\frac{c\left[(\mathfrak{z},\mathfrak{z})-2(\frac{x}{c},
\mathfrak{z})+(\frac{x}{c},\frac{x}{c})\right]
}{c\tau+d} +\frac{\frac{(x,x)}{c}+(a\tau
+b)(x,x)}{c\tau+d}&\\
&=-\frac{c(\mathfrak{z}-\frac{x}{c},\mathfrak{z}-\frac{x}{c})}{c\tau+d}+\frac{
(x,x)a(c\tau+d)}{c(c\tau +d)}&\\
&=-\frac{c(\mathfrak{z}-\frac{x}{c},\mathfrak{z}-\frac{x}{c})}{c\tau+d}+\frac{
a(x,x)}{c}&
 \end{aligned}\]
Since $k$ is even this implies
\begin{eqnarray}\label{Ausdruck fuer E_{k,m} stern}\begin{aligned}  
&E_{k,m}^{\ast}(\tau,\mathfrak{z})&\\&=\sum_{c=1}^{\infty}{c^{-k}}{\sum_{\substack{
d\in\Z\\(c,d)=1}}{ \sum_{x\in
L}{\left(\tau+\frac{d}{c}\right)^{-k}\exp\left(\pi\ii
m\left(-\frac{c(\mathfrak{z}-\frac{x}{c},\mathfrak{z}-\frac{x}{c})}{c\tau+d}
+\frac{
a(x,x)}{c}\right)\right)}}}&
  \end{aligned}
\end{eqnarray}
One immedeately sees that the substitutions $d\mapsto d+c$ and $x\mapsto
x+c\cdot\mathfrak{e}_{j}$ correspond to $\tau\mapsto \tau+1$ and
$\mathfrak{z}+\mathfrak{e}_{j}$ for all $1\leq j\leq N$. Therefore
we can write
\[\begin{aligned}  
E_{k,m}^{\ast}&=&&\sum_{c=1}^{\infty}{c^{-k}}{\sum_{\substack
{d\bmod c\\(c,d)=1}}{
\,\sum_{\substack{x\in L\\x 
\bmod c}}{\,\,\sum_{\substack{l\in\Z\\h\in
L}}{\left(\tau+\frac{d}{c}+l\right)^{-k} \exp\left(\pi\ii m\frac{
a(x,x)}{c}\right)}}}}&\\
&&&\qquad\qquad\qquad\qquad\qquad\times\exp\left(\pi\ii
m\left(-\frac{(\mathfrak{z}-\frac{x}{c}+h,\mathfrak{z}-\frac{x}{c}+h)}{\tau+
\frac{d}{c}+l}\right)\right)&\\
&=&&\sum_{c=1}^{\infty}{c^{-k}}{\sum_{\substack{
d\bmod c\\(c,d)=1}}{
\,\,\sum_{\substack{x\in L\\x 
\bmod
c}}{F_{k,m}\left(\tau+\frac{d}{c},
\mathfrak{z}-\frac{x}{c}\right)e_{c}\left(md^{-1}\frac{1}{2}(x,x)\right)}}}&
  \end{aligned}\]

where $e_{c}(a^{-1}b)=e^{2\pi\ii n/c}$ with $an\equiv b \pmod c$,
\[F_{k,m}(\tau,\mathfrak{z}):=\sum_{\substack{l\in\Z\\h\in
L}}{(\tau+l)^{-k}\exp\left(-\pi\ii
m\frac{(\mathfrak{z}+h,\mathfrak{z}+h)}{\tau+l}\right)}\] 
and $x\equiv y\bmod c$ means $x-y\in (c\Z)^{N}$. $F_{k,m}$ is periodic in
$\tau$ and $\mathfrak{z}$. Using the Poisson summation formula we obtain
\begin{equation}\label{Poisson Formel}
F_{k,m}(\tau,\mathfrak{z})=\sum_{\substack{l\in\Z\\h\in
L}}{\gamma(l,h)}e^{2\pi\ii l\tau}e^{\pi\ii (h,\mathfrak{z})}\end{equation}
where
\[\gamma(l,h)=\int_{\Im(\tau)=y}{\tau^{-k}e^{-2\pi\ii l\tau}\int_{
\Im(\mathfrak{z})=\nu}{e^{-\pi\ii\left(\frac{m(\mathfrak{z},\mathfrak{z})}{\tau}
+(h,\mathfrak{z})\right)}\mathrm{d}\mathfrak{z}}\,\mathrm{d}\tau} \]
with $y>0$ and $\nu\in\R^{N}$ arbitrary.

This yields 
\[ E_{k,m}^{\ast}=\sum_{\substack{l\in\Z\\h\in L}}{e_{k,m}(l,h)e^{2\pi\ii
l\tau}e^{\pi\ii (h,\mathfrak{z})}}\]
where
\[e_{k,m}(l,h)=\gamma(l,h)\,\sum_{c=1}^{\infty}{c^{-k}{\sum_{\substack{d \bmod
c\\(c,d)=1}}\sum_{\substack{x\in L\\x \bmod
c}}{e_{c}\left(d^{-1}\frac{m}{2}(x,x)-\frac{1}{2}(h,x)+ld\right)}}}
\]
By replacing $x$ by $d\cdot x$ a permutation of the summands in the inner sum is
performed. Then the inner summand reads $e_{c}(dQ(x))$ where
\[Q(x)=Q_{l,h}(x):=\frac{m(x,x)-(h,x)+2l}{2}\,.\] 
We use the identity
\[\begin{aligned}
&\sum_{\substack{d \bmod
c\\(c,d)=1}}{e_{c}(dN)}=\sum_{\substack{a|c\\a|N}}{\mu\left(\frac{c}{a}\right)a}
&\end{aligned}\]
where $\mu$ denotes Möbius function. We introduce the quantity
\[\Norm_{a}(Q)=\sharp\{x\bmod a\,|\,Q(x)\equiv 0 \bmod a\}\,.\]
Then we obtain the following expression for the inner sum
\[\sum_{\substack{d \bmod c\\(c,d)=1\\
x\bmod
c}}{e_{c}\left(dQ(x)\right)}=\sum_{a|c}{\mu\left(\frac{c}{a}
\right)a\sum_{\substack{x\bmod c\\Q(x)\equiv 0\bmod
a}}1}=\sum_{a|c}{\mu\left(\frac{c}{a}
\right) \frac{c^{N}}{a^{N-1}}\Norm_{a}(Q)}\,.\]
By virtue of the identity $\sum_{b=1}^{\infty}{\mu(b)b^{-s}}=\zeta(s)^{-1}$ we
obtain
\[\begin{aligned}
   \sum_{c=1}^{\infty}{c^{-k}\sum_{a|c}{\mu\left(\frac{c}{a}
\right)
\frac{c^{N}}{a^{N-1}}\Norm_{a}(Q)}}&=&&\sum_{b=1}^{\infty}{\sum_{a=1}^{\infty
}(ab)^{-k}\mu(b)ab^{N}\Norm_{a
} (Q)}\\
&=&&\zeta(k-N)^{-1}\cdot\sum_{a=1}^{\infty}{\frac{\Norm_{a}(Q)}{a^{k-1}}
} &\,.
  \end{aligned}\]
This gives us a formula for the Fourier expansion of $E_{k,m}$.
\begin{theorem}\label{theorem:Fourier expansion according to EZ}
 Let $k,m\in\N$ and $k>2+\operatorname{rank}(L)$ be even. Then $E_{k,m}$ is a
holomorhic function on $\mathcal{H}\times(L\otimes \C)$. Its Fourier expansion 
is given by 
\[E_{k,m}=\Theta_{mS,0,0}(\tau,\mathfrak{z})+\sum_{\substack{l\in\Z\\h\in
L}}{e_{k,m}(l,h)e^{2\pi\ii
l\tau}e^{\pi\ii (h,\mathfrak{z})}}\]
with Fourier-Jacobi coefficients
\[e_{k,m}(l,h)=\frac{\gamma(l,h)}{\zeta(k-N)}\,\sum_{a=1}^{\infty}{\frac{\Norm_{
a}(Q)}{a^{k-1}}
}\,.\]
Here the quantitity $\gamma$ is given by
\[\gamma(l,h)=\int_{\Im(\tau)=y}{\tau^{-k}e^{-2\pi\ii l\tau}\int_{
\Im(\mathfrak{z})=\nu}{e^{-\pi\ii\left(\frac{m(\mathfrak{z},\mathfrak{z})}{\tau}
+(h,\mathfrak{z})\right)}\mathrm{d}\mathfrak{z}}\,\mathrm{d}\tau} \]
with $y>0$ and $\nu\in\R^{N}$ arbitrary.
\end{theorem}
We want to give a slightly different description of the Fourier coefficients
where no integral has to be evaluated. This method makes use of the theta
transformation formula. For a reference confer for example \cite{KriegMFQuat}.
\begin{theorem}[Theta transformation
formula]\label{theorem:theta trafo}
Let $S\in\Z^{N\times N}$ be symmetric and positive definite and let
$a\in\Q$. Then the identity 
\[\Theta_{S,a,0}(-\tau^{-1},\mathfrak{z}\tau^{-1})=\left(\frac{\tau}{\ii}
\right)^ { N/2 } \left(\det
S\right)^{-1/2}e^{\pi\ii\frac{(\mathfrak{z},\mathfrak{z})_{S}}{\tau}}\,\sum_{
p\,:\,(S^{-1}\Z^{N})/\Z^{N}}{\Theta_{S,p,-a}(\tau,\mathfrak{z})} \]
holds for all $\tau\in\mathcal{H}$ and $\mathfrak{z}\in L\otimes \C$. 
\end{theorem}
\begin{proof}
 According to \cite{KriegMFQuat} we have
\[\begin{aligned}
\Theta_{S,a,0}(-\tau^{-1},\mathfrak{z}\tau^{-1})&=&&\Theta_{a,S\mathfrak{z}
\tau^{-1}}(-\tau^{-1},S,\Z^{N})&\\ 
&=&&\left(\frac{\tau}{\ii}
\right)^ { N/2 } \left(\det
S\right)^{-1/2}e^{2\pi\ii\frac{(a,\mathfrak{z})_{S}}{\tau}}\cdot
\Theta_{S\mathfrak{z}\tau^{-1},-a}(\tau,S^{-1},\Z^{N})&\\
&=&&\left(\frac{\tau}{\ii}
\right)^ { N/2 } \left(\det
S\right)^{-1/2}e^{2\pi\ii\frac{(a,\mathfrak{z})_{S}}{\tau}}\cdot
\Theta_{S^{-1},S\mathfrak{z}\tau^{-1},0}(\tau,-aS)\,.&
  \end{aligned}
\]
Expanding the last theta series yields
\[\begin{aligned}
   &\Theta_{S^{-1},S\mathfrak{z}\tau^{-1},0}(\tau,-aS)&\\&=\sum_{\lambda\in\Z^{N}}{
\operatorname{\exp}\left\{\pi\ii\left((\lambda+S\mathfrak{z}\tau^{-1},
\lambda+S\mathfrak{z}\tau^{-1})_{S^{-1}}\tau+2(\lambda+S\mathfrak{z}
\tau^{-1},-aS)_{S^{-1}}\right)\right\}} \,.&
  \end{aligned}\]
We calculate
\[(\lambda+S\mathfrak{z}\tau^{-1},
\lambda+S\mathfrak{z}\tau^{-1})_{S^{-1}}\tau=(\lambda,\lambda)_{S^{-1}}\,
\tau+2(\mathfrak{z},\lambda)_{I_{N}}+(\mathfrak{z},\mathfrak{z})_{S}\,\tau^{-1}
\]
and
\[2(\lambda+S\mathfrak{z}
\tau^{-1},-aS)_{S^{-1}}=-2(\lambda,a)_{I_{N}}-2(\mathfrak{z},a)_{S}\,\tau^{-1}\,
.\]
Therefore we obtain
\[\begin{aligned}
&\Theta_{S,a,0}(-\tau^{-1},\mathfrak{z}\tau^{-1})&\\&=\left(\frac{
\tau } { \ii }
\right)^ { N/2 } \left(\det
S\right)^{-1/2}e^{\pi\ii\frac{(\mathfrak{z},\mathfrak{z})_{S}}{\tau}}
\sum_{\lambda\in\Z^{N}}{\operatorname{\exp}\left\{
\pi\ii\left((S^{-1}\lambda,S^{-1}\lambda)_{S}\,\tau+2(S^{-1}\lambda,\mathfrak{z}
-a)_{S}\right)
\right\}}& \\
&=\left(\frac{\tau}{\ii}
\right)^ { N/2 } \left(\det
S\right)^{-1/2}e^{\pi\ii\frac{(\mathfrak{z},\mathfrak{z})_{S}}{\tau}}\,\sum_{
p\,:\,(S^{-1}\Z^{N})/\Z^{N}}{\Theta_{S,p,-a}(\tau,\mathfrak{z})}\,.&
\end{aligned}\] 
\end{proof}
Now we apply this formula to $E_{k,m}^{\ast}$. We consider the inner sum of
(\ref{Ausdruck fuer E_{k,m} stern}) and obtain by virtue of the theta
transformation formula 
\[\begin{aligned}
&\sum_{x\in
L}{\exp\left(\pi\ii
m\left(-\frac{c(\mathfrak{z}-\frac{x}{c},\mathfrak{z}-\frac{x}{c})}{c\tau+d}
+\frac{a(x,x)}{c}\right)\right)}&\\
&=\sum_{\substack{x\in
L\\x \bmod c}}{\exp\left(\pi\ii
m\frac{a(x,x)}{c}\right)\,\sum_{\lambda\in
L}{\exp\left(\pi\ii
m\left(-\frac{(\mathfrak{z}-\frac{x}{c}+\lambda,\mathfrak{z}-\frac{x}{c}
+\lambda) } { \tau+d/c }
\right)\right)}}&\\
&=\sum_{\substack{x\in
L\\x \bmod c}}{\exp\left(\pi\ii
m\frac{a(x,x)}{c}\right)\exp\left(-\pi\ii
m\frac{(\mathfrak{z}-\frac{x}{c},\mathfrak{z}-\frac{x}{c})}{\tau+d/c}\right)}&\\
&\qquad\times\Theta_{mS,0,0}\left(-(\tau+d/c)^{-1},(\mathfrak{z}-x/c)(\tau+d/c)^{-1}\right)
& \\
&=\left(\frac{\tau+d/c}{\ii}
\right)^ { N/2 } \left(\det
mS\right)^{-1/2}\sum_{\substack{x\in
L\\x \bmod c}}{\exp\left(\pi\ii
m\frac{a(x,x)}{c}\right)}&\\
&\qquad\times\sum_{
p\,:\,((mS)^{-1}\Z^{N})/\Z^{N}}{\Theta_{mS,p,0}(\tau+d/c,\mathfrak{z}-x/c)}&\\
&=\left(\frac{\tau+d/c}{m\ii}
\right)^ { N/2 } \left(\det
S\right)^{-1/2}\sum_{\substack{x\in
L\\x \bmod c}}{\exp\left(\pi\ii
m\frac{a(x,x)}{c}\right)}&\\
&\qquad\times\sum_{
p\,:\,((S)^{-1}\Z^{N})/\Z^{N}}{\Theta_{S,p,0}\left(\frac{\tau+d/c}{m},\mathfrak{
z } -x/c\right)}\,.&
  \end{aligned}	\]
Now we use division with remainder to write  $d=d'+lmc$. This yields
\[\begin{aligned}
   &E_{k,m}^{\ast}(\tau,\mathfrak{z})&\\&=\frac{(\det
S)^{-1/2}}{\ii^{N/2}}\sum_{\substack{c\in\N\\d' \bmod mc\\(c,d')=1}}{
(mc)^{-k}\sum_{l\in\Z}{\left(\frac{\tau}{m}+\frac{d'}{mc}+l\right)^{
N/2-k}}}&\\
&\times\sum_{\substack{x\in
L\\x \bmod c}}{\exp\left(\pi\ii
m\frac{a(x,x)}{c}\right)}\sum_{
p\,:\,((S)^{-1}\Z^{N})/\Z^{N}}{\Theta_{S,p,0}\left(\frac{\tau}{m}+\frac{d'}{mc}+
l, \mathfrak {
z }-x/c\right)}&\\
&=\frac{(\det
S)^{-1/2}}{\ii^{N/2}}\sum_{\substack{c\in\N\\d' \bmod mc\\(c,d')=1}}{
(mc)^{-k}\sum_{l\in\Z}{\left(\frac{\tau}{m}+\frac{d'}{mc}+l\right)^{
N/2-k}}}&\\
&\times\sum_{\substack{x\in
L\\x \bmod c}}{\exp\left(\pi\ii
m\frac{a(x,x)}{c}\right)}\sum_{
p\,:\,((S)^{-1}\Z^{N})/\Z^{N}}{e^{\pi\ii
l(p,p)}\Theta_{S,p,0}\left(\frac{\tau}{m} +\frac { d' } { mc } , \mathfrak {
z }-x/c\right)}&
  \end{aligned}\]
As usual we define the level of $L$  as
\[\min\left\{\mu\in\N\,|\,\mu(p,p)_{S}\in2\Z\,\,\text{for
all}\,\,p\in L^{\vee}\right\}=q\left(L\right)\,.\]
Now we can apply the well-known identity 
\[\sum_{l\in\Z}{\left(\tau+l\right)^{-k}}=\frac{\left(-2\pi\ii\right)^{k}}{
(k-1)! } \cdot\sum_{ r\in\N }{r^{ k-1}e^{2\pi\ii r\tau}} \]
which holds for all $\tau\in\mathcal{H}$ and $k\geq 2$ to obtain
\begin{eqnarray*}\begin{aligned}&E_{k,m}^{\ast}(\tau,\mathfrak{z})&\\&=\frac {(\det
S)^{-1/2}}{\ii^{N/2}}\sum_{\substack{c\in\N\\d' \bmod mc\\(c,d')=1}}{
(mc)^{-k}\sum_{\nu
\bmod q(L)}\sum_{\substack{x\in
L\\x \bmod c}}{\exp\left(\pi\ii
m\frac{a(x,x)}{c}\right)}}&\\
&\times\sum_{
p\,:\,((S)^{-1}\Z^{N})/\Z^{N}}{e^{\pi\ii
\nu(p,p)}\Theta_{S,p,0}\left(\frac{\tau}{m} +\frac { d' } { mc } , \mathfrak {
z }-x/c\right)}\sum_{l\in\Z}{\left(\frac{\tau}{m}+\frac{d'}{mc}
+\nu+q(L)l\right)^{
N/2-k}}&\\
&=\frac{\left(-2\pi\ii /q(L)\right)^{k-N/2}}{
(k-N/2-1)! }\cdot\frac{(\det
S)^{-1/2}}{\ii^{N/2}}\sum_{\substack{c\in\N\\d' \bmod mc\\(c,d')=1}}{
(mc)^{-k}\sum_{\nu
\bmod q(L)}\sum_{\substack{x\in
L\\x \bmod c}}{\exp\left(\pi\ii
m\frac{a(x,x)}{c}\right)}}&\\
&\times\sum_{
p\,:\,((S)^{-1}\Z^{N})/\Z^{N}}{e^{\pi\ii
\nu(p,p)}\Theta_{S,p,0}\left(\frac{\tau}{m} +\frac { d' } { mc } , \mathfrak {
z }-x/c\right)}&\\
&\times\sum_{ r\in\N }{r^{ k-N/2-1}\exp\left(2\pi\ii
r\left(\frac{\tau}{mq(L)}+\frac{d'}{mq(L)c}+
\frac{\nu}{q(L)}\right)\right)} &\\
  \end{aligned}\end{eqnarray*}
whenever $N$ is even. This gives us a second description of the Fourier
coefficients.
\begin{theorem}\label{theorem:Fourier expansion for Ekm ala Theta
Transformation formula}
  Let $N=\operatorname{rank}(L)$ be even and $k,m\in\N$ with 
$k>2+N$ even. Then
$E_{k,m}$ is a holomorhic Jacobi form on $\mathcal{H}\times(L\otimes \C)$.
If we set \(\Delta(n,\lambda)=nm-\frac{1}{2}(\lambda,\lambda)\) the Fourier
expansion of $E_{k,m}$ is given by 
\[E_{k,m}(\tau,\mathfrak{z})=\Theta_{mS,0,0}(\tau,\mathfrak{z})+\sum_{n=1}^{
\infty } { \sum_ { \substack{\lambda\in
S^{-1}\Z^{N}\\nm>\frac{1}{2}(\lambda,\lambda)}}{\beta_{k,m}(n,\lambda)e^{2\pi\ii
n\tau}e^{2\pi\ii (\lambda,\mathfrak{z})}}}\]
with Fourier-Jacobi coefficients
\[\begin{aligned}\beta_{k,m}(n,\lambda)&=&&\frac{\left(-2\pi\ii/
q(L)\right)^{k-N/2}}{(k-N/2-1)!}\cdot \frac{\left(\det S\right)^{-1/2}}{\ii
^{N/2}}\left(q(L)\Delta(n,\lambda)\right)^{k-N/2-1}\sum_{
\substack{c\in\N\\d' \bmod mc\\(c,d')=1}}{(mc)^{-k}}&\\
&&&\times\qquad\sum_{\substack{x\in L\\x
\bmod c}}{\exp\left\{\frac{\pi\ii}{c}\left(
ma(x,x)-2(\lambda,x)+2nd'\right)\right\}}&\\
&&&\times\qquad	\sum_{\nu \bmod q(L)}{\exp\left(2\pi\ii
nm\nu\right)\exp\left(\pi\ii\nu
(\lambda+L,\lambda+L)\right)\exp\left(-\pi\ii\nu(\lambda,
\lambda)\right)} \,.& 
\end{aligned} \]
\end{theorem}
\begin{proof}
From the definition it is clear that $E_{k,m}$ is invariant under the
substitution $\tau\mapsto \tau +1$. Since $\Theta_{mS,0,0}$ satisfies the same
invariant property $E_{k,m}^{\ast}$ possesses it either. 
Thus a comparison of the Fourier coefficients yields 
\[\exp\left(2\pi\ii
\frac{\frac{q(L)}{2}(\lambda,\lambda)+n}{mq(L)}\right)\beta_{k,m}(n,
\lambda)=\beta_ { k , m } (n,\lambda)\quad\text{ for all}\quad
n\in\N\,,\,\lambda\in L^{\vee} \,.\]
Therefore $\beta_{k,m}(n,\lambda)\neq 0$ implies
$\frac{q(L)}{2}(\lambda,\lambda)+n\equiv 0 \bmod mq(L)$. One finally observes
that the two sets
\[\left\{\left(\frac{\frac{q(L)}{2}(\lambda,\lambda)+n}{mq(L)},
\lambda\right)\,;\,\lambda\in L^{\vee}\,,\,n\in\N
\quad,\quad\frac{q(L)}{2}(\lambda, \lambda)+n\equiv 0 \bmod mq(L)\right\} \]
and
\[\left\{\left(n,\lambda\right)\,;\,n\in\N\,,\,\lambda\in
L^{\vee}\,,\, nm>\frac{1}{2}(\lambda,\lambda)\right\}\]
coincide. This yields the desired Fourier expansion.
\end{proof}
If we restrict ourselves to the case $m=1$ we can refine the this formula. We
define 
\[Q^{\vee}(x):=Q_{n,\lambda}(x):=\frac{m(x,x)-2(\lambda,x)+2n}{2}\quad\text{
where}\quad n\in\N\,,\,\lambda\in L^{\vee} \,.\]
Then with the same calculation as used before (\ref{theorem:Fourier expansion
according to EZ}) we obtain a more explicit result where no integral has to be
calculated.
\begin{theorem}\label{theorem: Fourier expansion for m equals 1}
   Let $N=\operatorname{rank}(L)$ be even and $k\in\N$ where
$k>2+N$ is even. Then $E_{k,1}$ is a holomorhic Jacobi form on
$\mathcal{H}\times(L\otimes \C)$. The Fourier expansion
of $E_{k,1}$ is given by 
\[E_{k,1}(\tau,\mathfrak{z})=\Theta_{S,0,0}(\tau,\mathfrak{z})+\sum_{n=1}^{
\infty } { \sum_ { \substack{\lambda\in
S^{-1}\Z^{N}\\n>\frac{1}{2}(\lambda,\lambda)}}{\beta_{k,1}(n,\lambda)e^{2\pi\ii
n\tau}e^{2\pi\ii (\lambda,\mathfrak{z})}}}\]
where
\[\beta_{k,1}(n,\lambda)=\frac{\Delta^{k-N/2-1}}{\operatorname{vol}(L)}\frac{i^{k-N}
(-2\pi)^{ k-N/2}}{\zeta(k-N)(k-N/2-1)!}\sum_{a=1}^{\infty}{
\frac{N_{a}(Q^{\vee})}{a^{k-1}}} \]
\end{theorem}
 
\section{An explicit formula for even unimodular lattices}\label{sec:L=E8}
We want to derive an explicit fomula for $e_{k,m}$ in this case. We will use several results from the theory of quadratic forms. Most of them are due to Siegel and can be found in his famous work ``\textit{Über die analytische Theorie der quadratischen Formen}'' \cite{Siegel}. Let $L$ again be an even integral positive definite lattice of rank $N$ and denote by $S$ the Gram matrix with respect to the standard basis of $\Z^{N}$. Then $L$ is a lattice in the $N$-dimensional $\Q$ vector space $V=L\otimes_{\Z}\Q$. For an integer $t$ we will write $L(t)$ for the lattice which arises from $L$ by rescaling the bilinear form by $t$. The associated quadratic form is given by
\[q\,:\,\Z^{N}\to\Z\quad,\quad q(x)=\frac{1}{2}(x,x)_{S}\,.\]
We write $|S|=\det S$. 
Now let $L'$ be another integral lattice in a finite dimensional $\Q$ vector space $V'$ with quadratic form $q'$. We say that $L$ is isometric to $L'$ if there exists a bijective $f\in\operatorname{Hom}_{\Z}(L,L')$ such that $q'(f(x))=q(x)$ for all $x\in L$. The notion of an isometry is readily generalized for $L$ being a module over a commutative ring with 1.
The \textit{orthogonal group of $L$} is given by
\[\operatorname{O}(L)=\{g\in\operatorname{GL}(L)\,|\,\forall x\in L:\,q(gx)=q(x)\}\,.\]
For a positive definite lattice $L$ the group $\operatorname{O}(L)$ turns out to be finite.

Let $p\in\mathbb{P}$ be a prime and denote by $\Q_{p}$ the p-adic completion of $\Q$ and by $\Z_{p}$ the integral p-adic numbers, i.e., the closure of $\Z$ in $\Q_{p}$. We write $\Q_{\infty}=\R$. By extending the coefficients of $V$ to $\Q_{p}$ we obtain the completions of $V$ with respect to the valuations $|\cdot|_{p}$ and denote them by $V_{p}=\Q_{p}\otimes_{\Q} V$ for $p\in\mathbb{P}\cup \{\infty\}$. Moreover for every place $p$ we define a $\Z_{p}$ lattice $L_{p}$ to be $L_{p}=\Z_{p}\otimes_{\Z} L$. Now we are able to introduce the notion of a genus. 
\begin{definition}
 The lattices $L$ and $L'$ belong to the same \textit{genus} if $L_{p}\cong L'_{p}$ for all $p\in\mathbb{P}\cup\{\infty\}$. This defines an equivalence relation on the set of all lattices in $V$. The \textit{genus of} $L$ is the equivalence class of $L$, i.e. \[\operatorname{genus}(L)=\{L'\,|\,\forall p\in\mathbb{P}\cup\{\infty\}:\,L_{p}\cong L'_{p}\}\,.\]
\end{definition} 
It is immedeately clear that isometric lattices belong to the same genus. If $L'$ is contained in the genus of $L$ the attached vector spaces $V$ and $V'$ are isomrphic. One can show that every genus contains only finitely many isometry classes of lattices. We will also refer to this quantity as the class number of $L$ written $\mathfrak{h}=\mathfrak{h}(L)$. A more detailed description of these facts can be found e.g. in \cite{Kneser}, chapter VII.

Let $L=L_{1},\dots,L_{\mathfrak{h}}$ be representatives of the classes of $L$ with quadratic forms $q_{i}$ and let $t$ be an integer. We denote by
\[r(t,L_{i})=\sharp\{x\in L\,|\,q_{i}(x)=t\}\]
the number of representations of $t$ by the quadratic form $q_{i}$. The \textit{mass} of the genus of $L$ is given by 
\[\operatorname{mass}(L)=\sum_{i=1}^{\mathfrak{h}(L)}{|\operatorname{O}(L_{i})|^{-1}}\]
and the weight of the class $L_{i}$ by 
\[w(L_{i})=|\operatorname{O}(L_{i})|^{-1}/\operatorname{mass}(L)\,.\]   
We define the number of representations of $t$ by the genus of $L$ as 
\[r(t,\operatorname{genus}(L))=\sum_{i=1}^{\mathfrak{h}(L)}{w(L_{i})r(t,L_{i})}\,.\]
This leads us to a special case of Siegel's main theorem on quadratic forms.
\begin{theorem}[Siegel]\label{theorem: Siegels Hauptsatz}
Let $L$ be an even positive definite lattice of rank $N\geq 2$ and $t$ be an integer. The number of representations of $t$ by the genus of $L$ is given by
\begin{equation}\label{mass formula}
 r(t,\operatorname{genus}(L))=\varepsilon_{N}\prod_{p\in\mathbb{P}\cup \{\infty\}}{\delta_{p}(t,L)}
\end{equation}
where $\varepsilon_{N}=1$ whenever $N>2$ and $\varepsilon_{2}=\frac{1}{2}$. 
The quantities $\delta_{p}(t,L)$ are called local densities and are defined as
\[\delta_{p}(t,L)=\lim_{U\to t}{\frac{\operatorname{vol}(q^{-1}(U))}{\operatorname{vol}(U)}}\]
where  $U$ ranges over $p$-adic neighbourhoods of $t$ and the volumes are taken with respect to the Haar measures on $\Z_{p}^{N}$ and $\Z_{p}$ respectively for $p\in\P$ and for $p=\infty$ the set $U$ ranges over real neighbourhoods of $t$ and the volumes are taken with respect to the Lebesgue measures on $\R^{N}$ and $\R$ respectively. The product in (\ref{mass formula}) converges absolutely for $N\geq 4$. 
\end{theorem}
\begin{rem}\label{remark: siegls main theorem}
\begin{enumerate}[(1)]
 \item The convergence of the product in (\ref{mass formula}) was analyzed in \cite{Siegel}, Hilfssatz 15. The domain of convergence can be enlarged. In the case $N=2$ we additionally have to require that $|S|$ is not a square and in the case $N=3$ the quantity $2t|S|$ is required not to be a square.
 \item The local densities in the formula above can be computed explicitely. Let us first consider the case $p\in\P$. We consider the lattice $L(2)$ and denote the associated quadratic form by $q_{(2)}$. We regard $q_{(2)}$ as a map $q_{(2)}\,:\,\Z_{p}^{N}\to\Z_{p}$. Let $B_{r}(t)$ be a $p$-adic ball around $t$ of radius $0<r<1$, i.e., $B_{r}(t)=t+p^{a}\Z_{p}$ for some $a\in\N$. The above ratio turns out to be
\[\begin{aligned}
\frac{\operatorname{vol}(q_{(2)}^{-1}(B_{r}(t)))}{\operatorname{vol}(B_{r}(t))}&=&&\sum_{x\in\operatorname{Rep}(\Z_{p}^{N}/(p^{a}\Z_{p})^{N})}\frac{\operatorname{vol}\left\{x+(p^{a}\Z_{p})^{N}\,|\,q_{(2)}(x)\in t+(p^{a}\Z_{p})^{N}\right\}}{p^{-a}}&\\
&=&&p^{a}\sum_{x\in\operatorname{Rep}(\Z_{p}^{N}/(p^{a}\Z_{p})^{N})}{\operatorname{vol}\left(x+(p^{a}\Z_{p})^{N}\right)\cdot\chi_{p^{a}\Z_{p}}(q_{(2)}(x)-t)}&\\
&=&&p^{a(1-N)}\sum_{x\in\operatorname{Rep}(\Z_{p}^{N}/(p^{a}\Z_{p})^{N})}{\chi_{p^{a}\Z}(q_{(2)}(x)-t)}&\\
&=&&p^{a(1-N)}\,\sharp\{x\in\Z^{N}/p^{a}\Z^{N}\,|\,q_{(2)}(x)\equiv t \bmod p^{a}\}\,.&   
  \end{aligned}\]
Here $\chi$ denotes the characteristic function of a set. In \cite{Siegel}, Hilfssatz 13, it was shown that this ratio  stabilizes for $a>2\operatorname{ord}_{p}(2t)$. Hence the local densities are given by the formula
\begin{equation}\label{abstract formula for local densities}
 \delta_{p}(t,L(2))=p^{a(1-N)}\,\sharp\{x\in\Z^{N}/p^{a}\Z^{N}\,|\,q_{(2)}(x)\equiv t \bmod p^{a}\}\quad
\end{equation}
for $a>2\operatorname{ord}_{p}(2t)$ and $p\in\P\,.$ If $p\neq 2$ we immediately see from (\ref{abstract formula for local densities}) that $\delta_{p}(2t,L(2))= \delta_{p}(t,L)$ and one has $2\delta_{2}(2t,L(2))=\delta_{2}(t,L)$.
\item\label{remark: siegls main theorem:exact formula}
If $p$ is not a divisor of $2|S|$ we have an exact formula for the local densities. Let $l_{p}=\operatorname{ord}_{p}(t)$ and write $t=p^{l_{p}}t_{\bar{p}}$. In this case we always have 
\[ \delta_{p}(t,L(2))=p^{a(1-N)}\,\sharp\{x\in\Z^{N}/p^{a}\Z^{N}\,|\,q_{(2)}(x)\equiv t \bmod p^{a}\}\quad\text{ for  }a>\operatorname{ord}_{p}(t)\,.\]

We assume that $N$ is even. We put \[\varepsilon(p,L)=\left(\frac{(-1)^{N/2}|S|}{p}\right)\,.\]
The local density is given by the formula
\[\delta_{p}(2t,L(2))=\left(1-\varepsilon(p,L)p^{-N/2}\right)\left(1+\varepsilon(p,L)p^{1-N/2}+\dots+\varepsilon(p,L)^{l_{p}}p^{l_{p}(1-N/2)}\right)\,.
\] 
If $N$ is odd we put 
\[\varepsilon(p,L,t)=\left(\frac{(-1)^{(N-1)/2}2|S|t_{\bar{p}}}{p}\right)\,.\]
If $l_{p}\equiv 1 \bmod 2$ the formula is
\[\delta_{p}(2t,L(2))=(1-p^{1-N})(1+p^{2-N}+\dots+p^{(2-N)(\frac{l_{p}-1}{2})})\]
and if $l_{p}\equiv 0 \bmod 2$ we have
\[\delta_{p}(2t,L(2))=(1-p^{1-N})\left(1+p^{2-N}+\dots+p^{(2-N)(\frac{l_{p}}{2}-1)}+\frac{p^{(2-N)\frac{l_{p}}{2}}}{1-\varepsilon(p,L,t)p^{\frac{1-N}{2}}}\right)\,.\]
If $l_{p}=0$ the formula is $\delta_{p}(t,L)=(1-p^{1-N})/(1-\varepsilon(p,L,t)p^{\frac{1-N}{2}})$.

 According to \cite{Siegel}, Hilfssatz 26 and (71), the density at the place $p=\infty$ is
\[\delta_{\infty}(t,L)=(2\pi)^{N/2}\Gamma(\textstyle{\frac{N}{2}})^{-1}t^{\frac{N}{2}-1}|S|^{-\frac{1}{2}}\,.\]
\end{enumerate}
\end{rem}
Let us fix an even unimodular lattice of rank $N$ such that $N\equiv 0 \bmod 8$. We can reformulate the statement of Theorem \ref{theorem: Fourier expansion for m equals 1} by using the values of the Riemann's zeta function for even integers. These are given by
\[\zeta(2m)=(-1)^{m+1}\frac{(2\pi)^{2m}}{2\cdot(2m)!}B_{2m}\]
where $B_{2m}$ denotes the Bernoulli numbers. Hence we can rewrite the coefficients from the theroem as
 \begin{equation}\label{Bernoulli Zahlen als Vorfaktoren}
  \beta_{k,1}=-\frac{2(k-\textstyle\frac{N}{2})}{B_{\displaystyle k\textstyle-\frac{N}{2}}}\cdot\frac{\zeta(k-\frac{N}{2})}{\zeta(k-N)}\,\Delta^{\textstyle k-\frac{N}{2}-1}\sum_{a\in\N}{\frac{N_{a}(Q^{\vee})}{a^{k-1}}}
 \end{equation}
for $k$ being even and large enough. If $L$ is unimodular we can improve the formula given by Siegel.
\begin{lemma}\label{lemma:calculation of the local densities by means of Gaussian sums}
  Let $S\in\operatorname{Sym}(N,\Z)$ be an even unimodular matrix, i.e, $\det S=1$. Put $R=\frac{1}{2}S$. Let $p$ be a prime and $\Delta\in\Z$ and let $l\in\N$ be a positive integer. Then we have the following formula for the representation number
\[p^{l(1-N)}\Norm_{p^{l}}(R,\Delta)=\begin{cases}
                                     (1-p^{-N/2})\frac{1-p^{(1+\ord_{p}(\Delta))(1-\frac{N}{2})}}{1-p^{1-\frac{N}{2}}}&\text{ if }l>\ord_{p}(\Delta)\\
				      p^{l(1-\frac{N}{2})}+(1-p^{-N/2})\frac{1-p^{l(1-\frac{N}{2})}}{1-p^{1-\frac{N}{2}}}&\text{ if }l\leq\ord_{p}(\Delta)\\
				      1-p^{-\frac{N}{2}}&\text{ if }\Delta\text{ and }p\text{ are coprime}
                                    \end{cases}
\]
\end{lemma}
The proof includes some rather long and elementary calculations using Gaussian sums, so we decided to omit it.
The next formula is a reorganization of Siegels main theorem and can be found in \cite{Iwaniec}, (11.74).
\begin{prop}\label{Prop:reorganizationof siegels formula by Iwaniec}
 Let $S$ be an even lattice of rank $N\geq 4$ and assume that $N$ is even. Then we have the formula
 \begin{equation}\label{Iwaniec's formula}
  r(\Delta,\operatorname{genus}(S))=\frac{\delta_{\infty}(\Delta,S)}{L(N/2,\chi_{4D})}\left(\sum_{a|\Delta}{\chi_{4D}(a)a^{1-\frac{N}{2}}}\right)\,\prod_{p|2D}{\delta_{p}(\Delta,S)}
  \end{equation}
  where $D=(-1)^{\frac{N}{2}}|S|$ and $\chi_{4D}(a)=\left(\frac{4D}{a}\right)$ is the quadratic character and $L(s,\chi_{4D})$ denotes the associated Dirichlet L-series.
\end{prop}
By virtue of these formulae we can state a sharper result for the Fourier coefficients.
\begin{theorem}
 \label{theorem:Fourier coefficients for L even unimodular}
 Let $L$ be an even unimodular lattice of rank $N$ and $k>2+N$ even and denote by $q$ the associated quadratic form. The Fourier expansion of $E_{k,1}$ is given by 
\[E_{k,1}(\tau,\mathfrak{z})=\Theta_{S,0,0}(\tau,\mathfrak{z})+\sum_{n=1}^{
\infty } { \sum_ { \substack{\lambda\in
L\\n>q(\lambda)}}{\beta_{k,1}(n,\lambda)e^{2\pi\ii
n\tau}e^{2\pi\ii (\lambda,\mathfrak{z})}}}\]
where 
\[\beta_{k,1}(n,\lambda)=-\frac{(2k-N)}{B_{k-\frac{N}{2}}}\,\sigma_{k-\frac{N}{2}-1}(\Delta)\,.\]
In particular $\beta_{k,1}$ is always rational and does only depend on the rank of $L$. 
\end{theorem}
\begin{proof}
 Since $N_{a}$ is multiplicative we have
 \[\sum_{a\in\N}{\frac{N_{a}(Q^{\vee})}{a^{k-1}}}=\prod_{p\in\P}{\sum_{\nu=0}^{\infty}{\frac{p^{\nu(1-N)}N_{p^{\nu}}(Q^{\vee})}{p^{\nu(k-N)}}}}\,.\]
 If $p$ is a prime not dividing $\Delta$ Lemma \ref{lemma:calculation of the local densities by means of Gaussian sums} yields
 \[\begin{aligned}
    \sum_{\nu=0}^{\infty}{\frac{p^{\nu(1-N)}N_{p^{\nu}}(Q^{\vee})}{p^{\nu(k-N)}}}=1-\delta_{p}(\Delta,L)+\frac{\delta_{p}(\Delta,L)}{1-p^{N-k}}&=&&\frac{1-p^{N-k}(1-\delta_{p}(\Delta,L))}{1-p^{N-k}}&\\&=&&\frac{1-p^{\frac{N}{2}-k}}{1-p^{N-k}}\,.&
   \end{aligned}\]
 If $(p,\Delta)>1$ the same Lemma leads to 
 \[\sum_{\nu=0}^{\infty}{\frac{p^{\nu(1-N)}N_{p^{\nu}}(Q^{\vee})}{p^{\nu(k-N)}}}=\delta_{p}(\Delta,L)R_{p}(\Delta)\]
where
\[\begin{aligned}
R_{p}(\Delta)&=&&\frac{p^{(\ord_{p}(\Delta)+1)(N-k)}}{1-p^{N-k}}&\\&+&&\frac{\frac{1-p^{(\ord_{p}(\Delta)+1)(N-k)}}{1-p^{N-k}}-\frac{1-p^{(\ord_{p}(\Delta)+1)(1+\frac{N}{2}-k)}}{1-p^{1+\frac{N}{2}-k}}}{1-p^{(\ord_{p}(\Delta)+1)(1-\frac{N}{2})}} &\\
&+&&\frac{(1-p^{1-\frac{N}{2}})(1-p^{(\ord_{p}(\Delta)+1)(1+\frac{N}{2}-k)})}{(1-p^{(\ord_{p}(\Delta)+1)(1-\frac{N}{2})})(1-p^{-\frac{N}{2}})(1-p^{1+\frac{N}{2}-k})}\,.&
  \end{aligned}
\]
 So the Dirichlet-series equals
 \[\begin{aligned}
    &\sum_{a\in\N}{\frac{N_{a}(Q^{\vee})}{a^{k-1}}}&\\
&=\frac{\zeta(k-N)}{\zeta(k-\frac{N}{2})}\,\prod_{\substack{p\in\P\\p|\Delta}}{\frac{1-p^{N-k}}{1-p^{\frac{N}{2}-k}}\delta_{p}(\Delta,L)R_{p}(\Delta)}&\\
    &=\frac{r\left(\Delta,\operatorname{genus}(L)\right)\Gamma(\frac{N}{2})}{(2\pi)^{\frac{N}{2}}\Delta^{\frac{N}{2}-1}}\,\frac{\zeta(k-N)}{\zeta(k-\frac{N}{2})}\,\prod_{\substack{p\in\P\\p|\Delta}}{\frac{1-p^{N-k}}{1-p^{\frac{N}{2}-k}}R_{p}(\Delta)}\cdot\prod_{\substack{p\in\P\\p\nmid\Delta}}{\delta_{p}(\Delta,L)^{-1}}&\\
    &=\frac{r\left(\Delta,\operatorname{genus}(L)\right)\Gamma(\frac{N}{2})}{(2\pi)^{\frac{N}{2}}\Delta^{\frac{N}{2}-1}}\,\frac{\zeta(k-N)\zeta(\frac{N}{2})}{\zeta(k-\frac{N}{2})}\,\prod_{\substack{p\in\P\\p|\Delta}}{\frac{(1-p^{N-k})(1-p^{-\frac{N}{2}})}{1-p^{\frac{N}{2}-k}}R_{p}(\Delta)}&
   \end{aligned}
\]
where we have used Siegels Theorem \ref{theorem: Siegels Hauptsatz}. If we apply (\ref{Bernoulli Zahlen als Vorfaktoren}) and the identity
\[\zeta\left(N/2\right)=-\frac{(2\pi)^{\frac{N}{2}}}{2\cdot(\frac{N}{2})!}B_{\frac{N}{2}}\]
we obtain
\[\beta_{k,1}(n,\lambda)=r\left(\Delta,\operatorname{genus}(L)\right)\frac{(\frac{2k}{N}-1)\cdot B_{\frac{N}{2}}}{B_{k-\frac{N}{2}}}\,\Delta^{k-N}\,\prod_{\substack{p\in\P\\p|\Delta}}{\frac{(1-p^{-\frac{N}{2}})(1-p^{N-k})}{1-p^{\frac{N}{2}-k}}\,R_{p}(\Delta)}\,.\]
We now use Proposition \ref{Prop:reorganizationof siegels formula by Iwaniec} to express $r\left(\Delta,\operatorname{genus}(L)\right)$. By assumption we have $D=(-1)^{N/2}|S|=1$ thus the quadratic character reads
\[\chi_{4D}(a)=\left(\frac{4}{a}\right)=\begin{cases}0&\text{ if }a \in 2\N\\
                                         1&\text{ if }a \in 2\N+1
                                        \end{cases}\]
such that
\[L(N/2,\chi_{4D})=\prod_{p\in\P_{>2}}(1-p^{-N/2})^{-1}=\zeta(N/2)\,(1-2^{-N/2})\] 
and
\[\sum_{a|\Delta}{\chi_{4D}(a)a^{1-N/2}}=\sigma_{1-N/2}(\Delta)\,\frac{1-2^{1-N/2}}{1-2^{(\ord_{2}(\Delta)+1)(1-N/2)}}
\]
Hence the number of representations of $\Delta$ by the genus of $L$ equals
\[\begin{aligned}
   r\left(\Delta,\operatorname{genus}(L)\right)&=&&\frac{\delta_{\infty}(\Delta,S)}{L(N/2,\chi_{4D})}\left(\sum_{a|\Delta}{\chi_{4D}(a)a^{1-\frac{N}{2}}}\right)\,\prod_{p|2D}{\delta_{p}(\Delta,S)}&\\
&=&&\frac{(2\pi)^{\frac{N}{2}}\Delta^{\frac{N}{2}-1}}{\Gamma(\frac{N}{2})}\,\frac{\sigma_{1-\frac{N}{2}}(\Delta)\,\delta_{2}(\Delta,S)}{\zeta(N/2)(1-2^{-N/2})}\,\frac{1-2^{1-N/2}}{1-2^{(\ord_{2}(\Delta)+1)(1-N/2)}}&\\
&=&&-\frac{N}{B_{N/2}}\,\Delta^{\frac{N}{2}-1}\,\sigma_{1-\frac{N}{2}}(\Delta)&
  \end{aligned}\]
This yields
\[\beta_{k,1}(n,\lambda)=-\frac{(2k-N)}{B_{k-\frac{N}{2}}}\,\Delta^{k-\frac{N}{2}-1}\,\sigma_{1-\frac{N}{2}}(\Delta)\,\prod_{\substack{p\in\P\\p|\Delta}}{\frac{(1-p^{-\frac{N}{2}})(1-p^{N-k})}{1-p^{\frac{N}{2}-k}}\,R_{p}(\Delta)}\,.\]
Now a direct computation shows that the identity
\[\prod_{\substack{p\in\P\\p|\Delta}}\frac{(1-p^{-\frac{N}{2}})(1-p^{N-k})}{1-p^{\frac{N}{2}-k}}\,R_{p}(\Delta)=\frac{\sigma_{k-\frac
{N}{2}-1}(\Delta)}{\sigma_{1-\frac{N}{2}}(\Delta)}\,\Delta^{1+\frac{N}{2}-k}\]
is valid which completes the proof.
\end{proof}

\section{The lattice $L=NA_{1}$}
We will now consider the next series of lattices. This is given by $N$ orthogonal copies of the root lattice $A_{1}$. The Gram matrix of $L$ is given by
\[S=\begin{pmatrix}
                2&0&\cdots&0\\0&2&\cdots&0\\\vdots&&\ddots&\vdots\\0&0&\cdots&2
               \end{pmatrix}\in\operatorname{Mat}(N,\Z)\,.
\]
 For this purpose let us recall the first description of the Fourier coefficients from Theorem \ref{theorem:Fourier expansion according to EZ} as
\[E_{k,m}=\Theta_{mS,0,0}(\tau,\mathfrak{z})+\sum_{\substack{n\in\Z\\\lambda\in
L}}{e_{k,m}(n,\lambda)e^{2\pi\ii
n\tau}e^{\pi\ii (\lambda,\mathfrak{z})}}\]
\[e_{k,m}(n,\lambda)=\frac{\gamma(n,\lambda)}{\zeta(k-N)}\,\sum_{a=1}^{\infty}{\frac{\mathfrak{D}_{a}(\lambda,-\Delta,NA_{1})}{a^{k-1}}
}\,.\]
The quantity $\gamma(n,\lambda)$ was given by a muti-dimensional integral. In our case it is not hard to evaluate it. As in \cite{Eichler-Zagier} this can be done by defoming the path of
integration and using Hankel's integral representation of $1/\Gamma$. The result is
\[\gamma(n,\lambda)=\begin{cases}
               0&\text{ if }\Delta\leq 0\\
		\frac{(-1)^{(3k/2)-N-1}\,\pi^{k-N/2}
}{2^{k-2}\,\Gamma(k-N/2)}\cdot \Delta^{k-N/2-1}&\text{ if
}\Delta> 0
              \end{cases}
\]  
where $\Delta=\Delta(n,\lambda)=4n-\frac{1}{2}(\lambda,\lambda)$  again denotes the hyperbolic norm. The reason for the additional factor 4 in $\Delta$ is the fact that the summation now ranges over all vectors $\lambda$ in the lattice instead of the dual lattice. But in this case the two quantities coincide since the Gram matrix of the dual lattice is 
\[S_{\left(NA_{1}\right)^{\vee}}=\begin{pmatrix}
                \frac{1}{2}&0&\cdots&0\\0&\frac{1}{2}&\cdots&0\\\vdots&&\ddots&\vdots\\0&0&\cdots&\frac{1}{2}
               \end{pmatrix}\,.
\]
The rest of this section is dedicated to the evaluation of the Dirchlet-series. We remind to the definition of $\mathfrak{D}_{a}(\lambda,\Delta,NA_{1})$ which was given by 
\[\mathfrak{D}_{a}(\lambda,\Delta,NA_{1})=\sharp\left\{x\bmod (a\Z)^{N}\,
;\,q(2x-\lambda)\equiv
\Delta \bmod 4a\right\}\]
By the mutiplicativity of $\mathfrak{D}_{a}(\Delta,NA_{1})$ we can reduce our discusion to $a$ being a prime power. 
The next lemma says that these quantities can be treated by investigating the local densities associated with $L=NA_{1}$. 
\begin{lemma}
 Let $p\in\P_{>2}$ be an odd prime and $\Delta=4n-\frac{1}{2}(\lambda,\lambda)\in\Z$  as above. For each $\lambda\in\Z^{N}$ and $l\in\N$ we have the identity
\[\mathfrak{D}_{p^{l}}(\lambda,-\Delta,NA_{1})=\frac{\omega(\lambda)}{\alpha(\lambda)}\, A_{N}(-\Delta,p^{l})\]
where
\[A_{N}(-\Delta,p^{l}):=\sharp\left\{x\bmod (p^{l}\Z)^{N}\,
;\,x_{1}^{2}+\dots+x_{N}^{2}\equiv
-\Delta \bmod p^{l}\right\}\]
and
\[\begin{aligned}
&\alpha(\lambda)=\sharp \left\{\sigma\in\{0,1\}^{N}\,;\,\sum_{i\in \{1,\dots N\}\backslash \{i\,:\,\lambda_{i}\equiv 0\bmod 2\}}{\sigma_{i}}\equiv \sum_{i\in \{i\,:\,\lambda_{i}\equiv 0\bmod 2\}}{\sigma_{i}}\pmod 4\right\}\,,&\\
&\omega(\lambda)=\sum_{\substack{j=1\\j\equiv -\Delta\bmod 4}}^{N}{\,\binom{N}{j}}\,.&   
  \end{aligned}\]
\end{lemma}

\begin{proof}
We first want to remark that $\omega(\lambda)$ is indeed well-defined since $\Delta$ depends only on $\lambda \bmod 4$. More precisely it suffices to consider the number $\sharp\{i\in\ \{1,\dots,N\}\,;\,\lambda_{i}\equiv 0 \bmod 2\}$ to determine the residue class of  $\lambda_{1}^{2}+\dots+\lambda_{N}^{2} \bmod 4$.  The proof is divided into two parts. In the first part we show $\alpha(\lambda)\mathfrak{D}_{p^{l}}(\lambda,-\Delta,NA_{1})=X$ and in the second part we prove $\omega(\lambda)A_{N}(-\Delta,p^{l})=X$
where
\[X=\sharp\left\{x\bmod (2p^{l}\Z)^{N}\,;\,x_{1}^{2}+\dots+x_{N}^{2}\equiv
 -\Delta \bmod 4p^{l}\right\}\,.\] 
\begin{enumerate}[(a)]
 \item We first note that for $x_{j},\lambda_{j}\in\Z$ we have
 \[\bigcup_{x_{j}:\Z/2p^{l}\Z}{(x_{j}+2p^{l}\Z)}=\bigcup_{x_{j}:\Z/p^{l}\Z}{(2x_{j}-\lambda_{j}+2p^{l}\Z)}\,\sqcup\,\bigcup_{x_{j}:\Z/p^{l}\Z}{(2x_{j}-\lambda_{j}-1+2p^{l}\Z)}\]
for $j=1,\dots,N$. Thus the formula
\[\begin{aligned}
   &&&\sharp\left\{x\bmod (2p^{l}\Z)^{N}\,;\,x_{1}^{2}+\dots+x_{N}^{2}\equiv
 -\Delta \bmod 4p^{l}\right\}&\\
&=&&\sum_{\sigma\in\{0,1\}^{N}}{\sharp\left\{x\bmod (p^{l}\Z)^{N}\,
;\,\sum_{j=1}^{N}(2x_{j}-(\lambda_{j}+\sigma_{j}))^{2}\equiv
-\Delta \bmod 4p^{l}\right\}}&
  \end{aligned}\]	
is correct. For each $\sigma$ we rewrite the summand as
\[\begin{aligned}
   S_{\sigma,\lambda}=&\sharp\left\{x\bmod (p^{l}\Z)^{N}\,
;\,4(x_{1}^{2}+\dots+x_{N}^{2})-4[x_{1}(\lambda_{1}+\sigma_{1})+\dots+x_{N}(\lambda_{N}+\sigma_{N})]\right.&\\
&\left.\equiv c(\sigma,\lambda)-4n \bmod 4p^{l}\right\}\,.&
  \end{aligned}\]
where
\[c(\sigma,\lambda)=-[\sigma_{1}(2\lambda_{1}+1)+\dots+\sigma_{N}(2\lambda_{N}+1)]\,.\]
Now $ S_{\sigma,\lambda}\neq 0$ implies $c(\lambda,\sigma)\equiv 0\bmod 4$.
Let us consider the set
\[A(\lambda)=\{\sigma\in\{0,1\}^{N}\,;\,c(\sigma,\lambda)\equiv 0\bmod 4\}\]
One immedeately sees that $c$ is periodic in $\lambda$ and more exactly we have $c(\sigma,\lambda+2)=c(\sigma,\lambda)$. Thus it suffices to consider $\lambda \bmod 2$. We put
\[J(\lambda)=\{i\in\ \{1,\dots,N\}\,;\,\lambda_{i}\equiv 0 \bmod 2\}\,.\] 
and observe that for each integer $a$ we have 
\[2a+1\equiv\begin{cases}
        1&\text{if } a\equiv 0 \bmod 2\\ 
        -1&\text{if } a\equiv 1 \bmod 2 
       \end{cases}\quad \pmod 4\,.\]
This yields
\[\begin{aligned}
c(\sigma,\lambda)\equiv -\sum_{i=1}^{N}{\sigma_{i}(2\lambda_{i}+1)} \equiv \sum_{i\in \{1,\dots N\}\backslash J(\lambda)}{\sigma_{i}}-\sum_{i\in J(\lambda)}{\sigma_{i}} \pmod 4   
  \end{aligned}
\]
and thus
\[A(\lambda)=\left\{\sigma\in\{0,1\}^{N}\,;\,\sum_{i\in \{1,\dots N\}\backslash J(\lambda)}{\sigma_{i}}\equiv \sum_{i\in J(\lambda)}{\sigma_{i}}\pmod 4\right\}\,.\]
Put $j(\lambda):=\sharp J(\lambda)$. Now for $0\leq r<4$ we have
\[\sharp\left\{\sigma\in\{0,1\}^{N}\,;\,\sum_{i\in J(\lambda)}{\sigma_{i}}\equiv r\bmod 4 \right\}=\sum_{\substack{i=1\\i\equiv \Delta\bmod 4}}^{j(\lambda)}{\,\binom{j(\lambda)}{i}}=\sum_{i=0}^{\lfloor \frac{j(\lambda)}{4}\rfloor}{\binom{j(\lambda)}{4i+r}}\]
This yields the following formula for $\alpha(\lambda)=\sharp A(\lambda)$
\begin{equation}\label{formula for alpha(lambda)}
 \alpha(\lambda)=\sum_{r=0}^{\min\{3,j(\lambda),N-j(\lambda)\}}\sum_{i=0}^{\lfloor \frac{j(\lambda)}{4}\rfloor}{\binom{j(\lambda)}{4i+r}}\cdot\sum_{i=0}^{\lfloor \frac{N-j(\lambda)}{4}\rfloor}{\binom{N-j(\lambda)}{4i+r}}
\end{equation}

Now let $u\in\Z$  such that $2u\equiv 1 \bmod p^{l}$. For each $\sigma\in A(\lambda)$ we have
\[\begin{aligned}
   S_{\sigma,\lambda}&=&&\sharp\left\{x\bmod (p^{l}\Z)^{N}\,
;\,(x_{1}^{2}+\dots+x_{N}^{2})-[x_{1}(\lambda_{1}+\sigma_{1})+\dots+x_{N}(\lambda_{N}+\sigma_{N})]\right.&\\
&&&\left.\equiv
u^{2}c(\sigma,\lambda)-n \bmod p^{l}\right\}&\\
&=&&\sharp\left\{x\bmod (p^{l}\Z)^{N}\,
;\,(x_{1}-u(\lambda_{1}+\sigma_{1}))^{2}+\dots+(x_{N}-u(\lambda_{N}+\sigma_{N}))^{2}\right.&\\
&&&\left.\equiv
u^{2}[\lambda_{1}^{2}+\dots+\lambda_{N}^{2}-4n] \bmod p^{l}\right\}&\\
&=&&\sharp\left\{x\bmod (p^{l}\Z)^{N}\,
;\,x_{1}^{2}+\dots+x_{N}^{2}\equiv
\lambda_{1}^{2}+\dots+\lambda_{N}^{2}-4n \bmod p^{l}\right\}&
  \end{aligned}
\]
since $u$ is a unit in $\Z/p^{l}\Z$.
Obviously $\sigma_{0}:=(0,\dots,0)$ is always an element of $A(\lambda)$ and we have
\[S_{\sigma_{0},\lambda}=\mathfrak{D}_{p^{l}}(\lambda,-\Delta,NA_{1})\,.\]
But the above calculations show that $S_{\sigma,\lambda}=S_{\sigma_{0},\lambda}=\mathfrak{D}_{p^{l}}(\lambda,-\Delta,NA_{1})$ for all $\sigma\in A(\lambda)$ such that 
\[\alpha(\lambda)\,\mathfrak{D}_{p^{l}}(\lambda,-\Delta,NA_{1})=\sharp\left\{x\bmod (2p^{l}\Z)^{N}\,;\,x_{1}^{2}+\dots+x_{N}^{2}\equiv
 -\Delta \bmod 4p^{l}\right\}\,.\]
\item One can use a similiar decomoposition as in (a) and obtains
\[\begin{aligned}
&\sharp\left\{x\bmod (2p^{l}\Z)^{N}\,;\,x_{1}^{2}+\dots+x_{N}^{2}\equiv
 -\Delta \bmod 4p^{l}\right\}&\\
&=\sum_{\sigma\in\{0,1\}^{N}}{\sharp\left\{x\bmod (p^{l}\Z)^{N}\,
;\,(2x_{1}+\sigma_{1})^{2}+\dots+(2x_{N}+\sigma_{N})^{2}\equiv
-\Delta \bmod 4p^{l}\right\}}\,.&   
  \end{aligned}
\]
We set 
\[\Omega(\lambda)=\{\sigma\in\{0,1\}^{N}\,|\,\sigma_{1}+\dots+\sigma_{N}\equiv-\Delta\bmod 4\}\,.\]
Note that  $\omega(\lambda)=\sharp\Omega(\lambda)$ is the number of all summands which equal $\Delta \bmod 4$ , i.e.
\[\omega(\lambda)=\sum_{\substack{j=1\\j\equiv -\Delta\bmod 4}}^{N}{\,\binom{N}{j}}=\sum_{j=0}^{\lfloor\frac{N}{4}\rfloor}{\binom{N}{4j+r}}\]
where we have written $-\Delta=4D+r$ such that $0\leq r<4$.  Let again $u\in\Z$ such that $2u\equiv1 \bmod p^{l}$. As above denote by $S_{\sigma}$ the summands which now only depend on $\sigma$. We observe that
\[S_{\sigma}\neq 0\implies \sigma_{1}+\dots+\sigma_{N}\equiv \Delta \bmod 4\]
and this yields
\[\begin{aligned}
   &\sharp\left\{x\bmod (2p^{l}\Z)^{N}\,;\,x_{1}^{2}+\dots+x_{N}^{2}\equiv
 -\Delta \bmod 4p^{l}\right\}&\\
&=\sum_{\sigma\in\Omega(\Delta)}{\sharp\left\{x\bmod (p^{l}\Z)^{N}\,
;\,(2x_{1}+\sigma_{1})^{2}+\dots+(2x_{N}+\sigma_{N})^{2}\equiv
-\Delta \bmod 4p^{l}\right\}}&\\
&=\sum_{\sigma\in\Omega(\Delta)}{\sharp\left\{x\bmod (p^{l}\Z)^{N}\,
;\,(x_{1}+\sigma_{1}u)^{2}+\dots+(x_{N}+\sigma_{N}u)^{2}\equiv
-u^{2}\Delta \bmod p^{l}\right\}}&\\
&=\omega(\Delta)\,\sharp\left\{x\bmod (p^{l}\Z)^{N}\,
;\,x_{1}^{2}+\dots+x_{N}^{2}\equiv
-u^{2}\Delta \bmod p^{l}\right\}\,.&
  \end{aligned}\]
Since $u$ is a unit in $\Z/p^{l}\Z$ we infer that 
\[\sharp\left\{x\bmod (p^{l}\Z)^{N}\,
;\,x_{1}^{2}+\dots+x_{N}^{2}\equiv
-u^{2}\Delta \bmod p^{l}\right\}\] 
equals
\[\sharp\left\{x\bmod (p^{l}\Z)^{N}\,
;\,x_{1}^{2}+\dots+x_{N}^{2}\equiv
-\Delta \bmod p^{l}\right\}\]
which finally yields
\[\omega(\lambda)\,A_{N}(-\Delta,p^{l})=\sharp\left\{x\bmod (2p^{l}\Z)^{N}\,;\,x_{1}^{2}+\dots+x_{N}^{2}\equiv
 -\Delta \bmod 4p^{l}\right\}\]
\end{enumerate}

\end{proof}

In fact one can show that the quotient $\frac{\omega(\lambda)}{\alpha(\lambda)}$ is always 1. Thus we obtain the following 
\begin{cor}\label{cor:coefficients of D series NA1 are local densities}
  Let $p\in\P_{>2}$ be an odd prime and $\Delta=4n-\frac{1}{2}(\lambda,\lambda)\in\Z$  as above. For each $\lambda\in\Z^{N}$ and $l\in\N$ we have the identity
\[\mathfrak{D}_{p^{l}}(\lambda,-\Delta,NA_{1})= A_{N}(-\Delta,p^{l})\,.\]
\end{cor}
\begin{proof}
 One shows that
\[\sum_{r=0}^{\min\{3,j(\lambda),N-j(\lambda)\}}\sum_{i=0}^{\lfloor \frac{j(\lambda)}{4}\rfloor}{\binom{j(\lambda)}{4i+r}}\sum_{k=0}^{\lfloor \frac{N-j(\lambda)}{4}\rfloor}{\binom{N-j(\lambda)}{4k+r}}=\sum_{j=0}^{\lfloor\frac{N}{4}\rfloor}{\binom{N}{4j+r}}\]
by an elementary computation.
\end{proof}


The case $p=2$ has to be treated separetely. Here we make use of the simple fact that
\[\mathfrak{D}_{2^{l}}(\lambda,-\Delta,NA_{1})=\sharp\left\{x\bmod (2^{l}\Z)^{N}\,
;\,(x_{1}^{2}-\lambda_{1}x_{1})+\dots+(x_{N}^{2}-\lambda_{N}x_{N})\equiv
-n \,(2^{l})\right\}\,.\] 
Now we can use some facts about Gaussian sums with a half integral shift which can be found in \cite{Berndt-Evans-Williams}, e.g.
The following two Propositions contain explicit formulas for the $\mathfrak{D}_{p^{l}}(\lambda,-\Delta,NA_{1})$. The proofs make use of generalized Gaussian sums. Since the calculations needed here are very tedious especially in the case $p=2$ we decided to omit them at this point. We start with the formulae for odd $p$.
\begin{prop}\label{prop:local densities for sum of N squares and p odd prime}
 Let $p$ be an odd prime and let $l\in\N$ be a positive integer. Assume that $N$ is even. Then we have the following formula for the local densities $p^{l(1-N)}A_{N}(-\Delta,p^{l})$
\[\begin{cases}
                                     (1-\varepsilon(p,NA_{1})\,p^{-N/2})\frac{1-\varepsilon(p,NA_{1})^{1+\ord_{p}(\Delta)}\,p^{(1+\ord_{p}(\Delta))(1-\frac{N}{2})}}{1-\varepsilon(p,NA_{1})\,p^{1-\frac{N}{2}}}&\text{ if }l>\ord_{p}(\Delta)\\
				      \varepsilon(p,NA_{1})^{l}\,p^{l(1-\frac{N}{2})}+(1-\varepsilon(p,NA_{1})\,p^{-N/2})\frac{1-\varepsilon(p,NA_{1})^{l}\,p^{l(1-\frac{N}{2})}}{1-\varepsilon(p,NA_{1})\,p^{1-\frac{N}{2}}}&\text{ if }l\leq\ord_{p}(\Delta)\\
				      1-\varepsilon(p,NA_{1})\,p^{-\frac{N}{2}}&\text{ if }(\Delta,p)=1
                                    \end{cases}\]
where we have used the abbreviation
\[\varepsilon(p,NA_{1})=\left(\frac{(-1)^{N/2}\det (NA_{1})}{p}\right)=\left(\frac{(-1)^{N/2}}{p}\right)\,.\]
\end{prop}
The next Proposition treats the case $p=2$ which turns out to be more complicated.
\begin{prop}\label{prop:local densities for sum of N squares and p equals 2}
  Let $N$ be an even integer and let $l\in\N$. We write
\[\Delta=4n-\frac{1}{2}(\lambda,\lambda)\]
\begin{enumerate}[(a)]
 \item If not all $\lambda_{k}$ are equivalent 
$\bmod\,2$ we have 
\[2^{l(1-N)}\mathfrak{D}_{2^{l}}(\lambda,-\Delta,NA_{1})=1\,.\]
 \item If all $\lambda_{k}\equiv 1 \bmod 2$ we have
\[2^{l(1-N)}\mathfrak{D}_{2^{l}}(\lambda,-\Delta,NA_{1})=1+(-1)^{-n}\,.\]
 \item Put $\kappa:=-\Delta/4$. Assume that all $\lambda_{k}\equiv 0 \bmod 2$. If $2$ and $\kappa$ are coprime then $2^{l(1-N)}\mathfrak{D}_{2^{l}}(\lambda,-\Delta,NA_{1})$ equals
\[\begin{cases}
                                    1-(-1)^{\frac{N}{4}+\frac{\kappa}{2}}2^{1-\frac{N}{2}}&\text{ if }l\geq 2 \text{ and }N\equiv 2\bmod 4\\
				     1&\text{ if } l=1\quad\text{ or }\quad l\geq 2\text{ and }N\equiv 0\bmod 4
                                   \end{cases}\]
If $l\leq\ord_{2}(\kappa)$ then $2^{l(1-N)}\mathfrak{D}_{2^{l}}(\lambda,-\Delta,NA_{1})$ equals
\[\begin{cases}
                                            1-(-1)^{\frac{N}{4}}\,\,\frac{1-2^{(l-1)(1-\frac{N}{2})}}{1-2^{\frac{N}{2}-1}}&\text{ if }N\equiv 0\bmod 4\\
					    1 &\text{ if }N\equiv 2 \bmod 4
                                        \end{cases}\]
If $l=\ord_{2}(\kappa)+1$ then $2^{l(1-N)}\mathfrak{D}_{2^{l}}(\lambda,-\Delta,NA_{1})$ equals
\[\begin{cases}
                                          1 &\text{ if }N\equiv 2\bmod 4\\
					  1+(-1)^{\frac{N}{4}}2^{\ord_{2}(\kappa)(1-\frac{N}{2})}\left\{\frac{2^{\ord_{2}(\kappa)(\frac{N}{2}-1)}-1}{2^{\frac{N}{2}-1}-1}-2\right\}&\text{ if }N\equiv 0\bmod 4
                                         \end{cases}\]
and if $l\geq\ord_{2}(\kappa)+2$ then $2^{l(1-N)}\mathfrak{D}_{2^{l}}(\lambda,-\Delta,NA_{1})$ equals
\[\begin{cases}
                                    1-(-1)^{\frac{N}{4}+\frac{\kappa_{\bar{2}}}{2}} 2^{(\ord_{2}(\kappa)+1)(1-\frac{N}{2})}&\text{ if }N\equiv 2 \bmod 4\\
				    1-(-1)^{\frac{N}{4}}\left(\frac{1-2^{(\ord_{2}(\kappa)-1)(1-\frac{N}{2})}}{1-2^{\frac{N}{2}-1}}+2^{\ord_{2}(\kappa)(1-\frac{N}{2})}\right)&\text{ if }N\equiv 0\bmod 4
                                   \end{cases}\,.\]
Here we have written $\kappa=2^{\ord_{2}(\kappa)}\,\kappa_{\bar{2}}$.
\end{enumerate}

\end{prop}

The next theorem gives a first description of the Fourier expansion. Here we have omitted an explicit evaluation of the local density for $p=2$.

\begin{theorem}
 Let $N\geq 4$ be even and assume that $L$ equals $NA_{1}$. We put
\[\alpha_{2}=\sum_{\nu=0}^{\infty}{\frac{\mathfrak{D}_{2^{\nu}}(\lambda,-\Delta,NA_{1})}{2^{\nu(k-1)}}}\]
 The Fourier expansion of the Jacobi-Eisenstein $E_{k,1}$ is given by
\[E_{k,1}=\Theta_{S,0,0}(\tau,\mathfrak{z})+\sum_{\substack{n\in\Z\\\lambda\in
L}}{e_{k,1}(n,\lambda)e^{2\pi\ii
n\tau}e^{\pi\ii (\lambda,\mathfrak{z})}}\]
where
\[\begin{aligned}
   e_{k,1}(n,\lambda)&=&&\alpha_{2}(1-2^{N-k})\,\frac{2^{2-N/2}(-1)^{\frac{N}{4}-\frac{1}{2}}}{(\frac{k}{2}-\frac{N}{4}-\frac{1}{2})!}\,\,\frac{\sum_{a|\Delta}{\chi_{4D}(a)a^{1-\frac{N}{2}}}}{L(1+N/2-k,\chi_{4D})}\,\Delta^{k-\frac{N}{2}}&\\&\times&&\prod_{\substack{p\in\P_{>2}\\p|\Delta}}{\frac{(1-\chi_{D}(p)p^{-\frac{N}{2}})(1-p^{N-k})}{1-\chi_{D}(p)p^{\frac{N}{2}-k}}\,R_{p}(\Delta)}\,.&
  \end{aligned}\]
if $N\equiv 2 \bmod 4$ and
\[\begin{aligned}
   e_{k,1}(n,\lambda)&=&&\alpha_{2}(1-2^{N-k})\,2^{\ord_{2}(\Delta)(k-N/2-1)}\,\sigma_{k-\frac{N}{2}-1}(\Delta)\,\Delta&\\
  &\times&&\frac{\sigma_{1-N/2}(2^{\ord_{2}(\Delta)})}{\sigma_{k-\frac{N}{2}-1}(2^{\ord_{2}(\Delta)})}\,\frac{(-1)^{N/4}\,(2k-N)}{2^{k-2-N/2}B_{k-\frac{N}{2}}}&
  \end{aligned}\]
if $N\equiv 0 \bmod 4$.
The quantity $R_{p}(\Delta)$ is defined by equation (\ref{Rp fuer L=NA1}) below and 
\[\chi_{4D}(a)=\left(\frac{4D}{a}\right)\quad,\quad D=(-1)^{\frac{N}{2}}\,.\]
\end{theorem}

\begin{proof}
As in the proof of Theorem \ref{theorem:Fourier coefficients for L even unimodular} we can write the Dirichlet series which occurs in the description as an Euler product and reduce the discussion to the local factors.
We first assume that $p$ is an odd prime. If $p\nmid \Delta$ we obtain from Corollary \ref{cor:coefficients of D series NA1 are local densities} and Proposition \ref{prop:local densities for sum of N squares and p odd prime}
\[\begin{aligned}
   \sum_{\nu=0}^{\infty}{\frac{\mathfrak{D}_{p^{\nu}}(\lambda,-\Delta,NA_{1})}{p^{\nu(k-1)}}}=\sum_{\nu=0}^{\infty}{\frac{p^{\nu(1-N)}A_{N}(-\Delta,p^{\nu})}{p^{\nu(k-N)}}}&=&&1-\delta_{p}(\Delta,NA_{1})+\frac{\delta_{p}(\Delta,NA_{1})}{1-p^{N-k}}&\\&=&&\frac{1-\chi_{D}(p)\,p^{\frac{N}{2}-k}}{1-p^{N-k}}&
  \end{aligned}\]
where we have put
\[\chi_{D}(p)=\left(\frac{D}{p}\right)=\left(\frac{4D}{p}\right)\quad,\quad D=(-1)^{\frac{N}{2}}\,.\]
If $p|\Delta$ we have to take into account the formula for bad reduction given by the same Proposition which yields
\[\sum_{\nu=0}^{\infty}{\frac{\mathfrak{D}_{p^{\nu}}(\lambda,-\Delta,NA_{1})}{p^{\nu(k-1)}}}=\delta_{p}(\Delta,NA_{1})R_{p}(\Delta)\]
where
\begin{equation}\label{Rp fuer L=NA1}
\begin{aligned}
&R_{p}(\Delta)=\frac{p^{(\ord_{p}(\Delta)+1)(N-k)}}{1-p^{N-k}}&\\&+\frac{\displaystyle \frac{1-p^{(\ord_{p}(\Delta)+1)(N-k)}}{1-p^{N-k}}-\frac{1-(\chi_{D}(p)p^{1+\frac{N}{2}-k})^{\ord_{p}(\Delta)+1}}{1-\chi_{D}(p)p^{1+\frac{N}{2}-k}}}{1-\chi_{D}(p)^{1+\ord_{p}(\Delta)}p^{(\ord_{p}(\Delta)+1)(1-\frac{N}{2})}} &\\
&+\frac{(1-\chi_{D}(p)p^{1-\frac{N}{2}})[1-(\chi_{D}(p)p^{1+\frac{N}{2}-k})^{\ord_{p}(\Delta)+1}]}{(1-\chi_{D}(p)^{1+\ord_{p}(\Delta)}p^{(\ord_{p}(\Delta)+1)(1-\frac{N}{2})})(1-\chi_{D}(p)p^{-\frac{N}{2}})(1-\chi_{D}(p)p^{1+\frac{N}{2}-k})}\,.&
  \end{aligned} 
\end{equation}

If we combine these two expressions we can perform manipulations similiar to those in the proof of Theorem \ref{theorem:Fourier coefficients for L even unimodular}. This yields
\[\begin{aligned}
   \sum_{a\in\N}{\frac{\mathfrak{D}_{a}(\lambda,-\Delta,NA_{1})}{a^{k-1}}}&=&&\frac{\alpha_{2}\,(1-2^{N-k})}{\delta_{2}(\Delta,NA_{1})}\,\frac{r(\Delta,\operatorname{genus}(NA_{1}))}{\delta_{\infty}(\Delta,NA_{1})}\,\frac{\zeta(k-N)L(N/2,\chi_{4D})}{L(k-N/2,\chi_{4D})}&\\
&\times&&\prod_{\substack{p\in\P_{>2}\\p|\Delta}}{\frac{(1-\chi_{D}(p)p^{-\frac{N}{2}})(1-p^{N-k})}{1-\chi_{D}(p)p^{\frac{N}{2}-k}}\,R_{p}(\Delta)}\,.&
  \end{aligned}\]
If we apply Iwaniec's formula \ref{Iwaniec's formula} we obtain
\[\begin{aligned}
   &e_{k,1}(n,\lambda)=\alpha_{2}(1-2^{N-k})\,\frac{(-1)^{\frac{k}{2}-1}\pi^{k-\frac{N}{2}}\Delta^{k-\frac{N}{2}}}{2^{k-2}\,\Gamma(k-\frac{N}{2})}\,\,\frac{\sum_{a|\Delta}{\chi_{4D}(a)a^{1-\frac{N}{2}}}}{L(k-N/2,\chi_{4D})}&\\&\times\prod_{\substack{p\in\P_{>2}\\p|\Delta}}{\frac{(1-\chi_{D}(p)p^{-\frac{N}{2}})(1-p^{N-k})}{1-\chi_{D}(p)p^{\frac{N}{2}-k}}\,R_{p}(\Delta)}\,.&
  \end{aligned}\]
The functional equation of the classical $L$-function yields in the case where $N\equiv 2 \bmod 4$
\[L(k-N/2,\chi_{4D})=\pi^{k-\frac{N}{2}-\frac{1}{2}}\,4^{\frac{1+N-2k}{2}}\,\frac{\Gamma(1+\frac{N}{4}-\frac{k}{2})}{\Gamma(\frac{k}{2}-\frac{N}{4}+\frac{1}{2})}\,L(1+N/2-k,\chi_{4D})\,.\]
Now by the functional equation of the $\Gamma$-function we obtain
\[\Gamma\left(1+\frac{N}{4}-\frac{k}{2}\right)=\frac{\sqrt{\pi}\,2^{k-\frac{N}{2}-1}}{(-1)^{\frac{1}{2}+\frac{N}{4}-\frac{k}{2}}}\,\frac{\Gamma(\frac{k}{2}-\frac{N}{4}+\frac{1}{2})}{\Gamma(k-\frac{N}{2})}\]
So if $N\equiv 2 \bmod 4$ we have
\[\begin{aligned}
   &e_{k,1}(n,\lambda)=\alpha_{2}(1-2^{N-k})\,\frac{2^{2-N/2}(-1)^{\frac{N}{4}-\frac{1}{2}}}{(\frac{k}{2}-\frac{N}{4}-\frac{1}{2})!}\,\,\frac{\sum_{a|\Delta}{\chi_{4D}(a)a^{1-\frac{N}{2}}}}{L(1+N/2-k,\chi_{4D})}\,\Delta^{k-\frac{N}{2}}&\\&\times\prod_{\substack{p\in\P_{>2}\\p|\Delta}}{\frac{(1-\chi_{D}(p)p^{-\frac{N}{2}})(1-p^{N-k})}{1-\chi_{D}(p)p^{\frac{N}{2}-k}}\,R_{p}(\Delta)}\,.&
  \end{aligned}\]
If $N\equiv 4 \bmod 4$ the same computation as at the end of the proof of Theorem \ref{theorem:Fourier coefficients for L even unimodular} shows 
\[\begin{aligned}
   e_{k,1}(n,\lambda)&=&&\alpha_{2}(1-2^{N-k})\,2^{\ord_{2}(\Delta)(k-N/2-1)}\,\frac{\sigma_{1-N/2}(2^{\ord_{2}(\Delta)})}{\sigma_{k-\frac{N}{2}-1}(2^{\ord_{2}(\Delta)})}&\\&\times&&\frac{(-1)^{N/4}\,(2k-N)}{2^{k-2-N/2}B_{k-\frac{N}{2}}}\,\sigma_{k-\frac{N}{2}-1}(\Delta)\,\Delta\,.&
  \end{aligned}\]
\end{proof}


\begin{thebibliography}{10}

\bibitem[BEW]{Berndt-Evans-Williams}B.C. Berndt, R.J. Evans,K.S. Williams. \emph{Gauss and Jacobi
Sums}, Canadian Mathematical Society, Series of monographs and advanced texts,
vol. 21 (1998)
\bibitem[EK]{Eie-Krieg}M. Eie, A. Krieg. \emph{The theory of Jacobi forms
over the Cayley numbers}, Transactions of the american mathematical society,
vol. \textbf{342}, No.2 (1994) 
\bibitem[EZ]{Eichler-Zagier}M. Eichler, D. Zagier. \emph{The Theory of Jacobi
Forms} Birkhäuser,
Boston/Basel/Stuttgart, 1985.
\bibitem[G]{Gritsenko}V. Gritsenko \emph{Modular forms and moduli spaces of
abelian and $K3$ surfaces}, Mathematica Gottingensis, Heft 26,1993
\bibitem[Iw]{Iwaniec}H. Iwaniec, \emph{Topics in classical automorphic forms}, Graduate Studies in Math. \textbf{17}, Amer. Math. Soc., 1997.
\bibitem[K1]{KriegMFQuat}A. Krieg, \emph{Modular forms on half-spaces of
quaternions}, Lecture Notes in Math., vol. 1143, Springer-Verlag,
Berlin--Heidelberg--New York, 1989
\bibitem[Kn]{Kneser}M. Kneser. \emph{Quadratische Formen}, Springer Verlag, 2001.
\bibitem[Si]{Siegel}C.L. Siegel, \emph{Über die analytische Theorie der quadratischen Formen}, Ann. of Math. 36 (1935), 527--606
\bibitem[Za]{Zagier Modular forms zeta functions}D. Zagier, \emph{Modular forms
whose Fourier coefficients involve zeta-functions of quadratic fields}, Modular
Functions of one Variable, VI. Lecture Notes No. 627, Springer-Verlag,
Berlin--Heidelberg--New York (1977) 105--169
\bibitem[Z]{Ziegler}C. Ziegler, \emph{Jacobi forms of higer degree}, Abh. Math.
Sem. Univ. Hamburg (\textbf{59}) (1989), 191--224
\end{thebibliography}
\end{document}